\documentclass[smallextended]{sn-jnl}
\usepackage[english]{babel}
\usepackage{
    orcidlink,
    amsmath,
    amssymb,
    amsthm,
    amsfonts,
    mdframed,
    xcolor,
    caption,
    subcaption,
    comment, 
    mathrsfs,
    enumitem,
    float,
    graphicx,
    epstopdf,
    multirow,
    textcomp,
    manyfoot,
    booktabs,
    algorithm,
    algorithmicx,
    algpseudocode,
    listings,
    accents,
    sidecap
}

\title{A-priori error estimation for space-time Galerkin POD for linear evolution problems}

\author[1]{\fnm{Carmen} \sur{Gr\"a\ss le} \orcidlink{0000-0003-0318-0740}}\email{c.graessle@tu-braunschweig.de}
\author[2]{\fnm{Jan} \sur{Heiland} \orcidlink{0000-0003-0228-8522}}\email{jan.heiland@tu-ilmenau.de}
\author*[1]{\fnm{Jannis} \sur{Marquardt} \orcidlink{0009-0008-0248-7533}}\email{j.marquardt@tu-braunschweig.de}

\affil[1]{\orgdiv{Institute for Partial Differential Equations}, \orgname{Technische Universität Braunschweig}, \orgaddress{\street{Universit\"atsplatz 2}, \city{Braunschweig}, \postcode{38106}, \country{Germany}}}

\affil[2]{\orgdiv{Department of Mathematics and Natural Sciences, Institute of Mathematics}, \orgname{Technische
Universit\"at Ilmenau}, \orgaddress{\street{Weimarer Str. 25}, \city{Ilmenau}, \postcode{98693}, \country{Germany}}}

\theoremstyle{plain}% Theorem-like structures provided by amsthm.sty
\newtheorem{theorem}{Theorem}[section]
\newtheorem{lemma}[theorem]{Lemma}

\newtheorem{proposition}[theorem]{Proposition}
\newtheorem{remark}[theorem]{Remark}

\newtheorem{example}[theorem]{Example}

\newcommand{\R}{\mathbb{R}}

\newcommand{\M}{\mathbf{M}}
\newcommand{\K}{\mathbf{K}}

\newcommand{\A}{\mathcal{A}}

\newcommand{\Proj}{\mathcal{P}}

\newcommand{\Y}{\mathcal{Y}}

\renewcommand{\S}{\mathcal{S}}

\renewcommand{\L}{\mathcal{L}}

\newcommand{\intd}{\,\text{d}}

\newcommand{\braket}[2]{
\left\langle #1, #2 \right\rangle
}

\begin{document}

\abstract{
    In this paper, we propose an a-priori error estimate for the model order reduction (MOR) method of space-time proper orthogonal decomposition (space-time POD). The original space-time POD approach extends standard POD by reducing not only the space dimension but simultaneously the time dimension as well. The proposed a-priori error estimate is developed for a linear parabolic partial differential equation and estimates the error between the numerical solution to a linear parabolic partial differential equation (PDE) and its space-time POD reduced solution. Numerical examples illustrate the occurring errors and analyze them in comparison to the theoretical bounds.} 

 

\maketitle

\section*{Novelty statement}
In this paper, we consider the space-time proper orthogonal decomposition (space-time POD) introduced in \cite{baumann18} and extend the results by proposing an a-priori error bound for the error between a solution to a full order model, namely the numerical solution to a linear parabolic partial differential equation, and the solution to the corresponding reduced order model. In detail, this work contains the following novelty:
\begin{itemize}
    \item Lemma~\ref{lem:rewrite_x} extends \cite[Corollary~2.5]{baumann18}
    \item In Section~\ref{sec:space_time_pod_initial}, a new procedure to include the initial condition in the space-time POD context is proposed. 

    \item Lemma~\ref{lem:non-orthogonal-coeff-mat-connection} contains the computation of a reduced coefficient matrix in the special case, where the first column of the full order coefficient matrix has been altered and set to zero.

    \item The main contribution of this paper is the derivation of an a-priori error estimate for space-time POD in Proposition~\ref{prop:rho_error}, Proposition~\ref{prop:bound_evolution_problem}, which is similar to well known a-priori error estimates for standard POD. 

    \item For the numerical examples in Section~\ref{sec:numerical examples}, the measurement matrix is computed using a tensor product finite element method instead of a semi-discretization in space as done within the introduction of space-time POD in \cite{baumann18}.
\end{itemize}

\section{Introduction}
    \subsection{Problem formulation}
    Since decades, proper orthogonal decomposition (POD) has been a well established model order reduction (MOR) technique with applications including e.g.\ linear and nonlinear parabolic equations \cite{kunisch01}, optimal control of partial differential equations \cite{gubisch17, hinze08, kunisch99, arian00} or fluid dynamics \cite{hinze00, lassila14, lumley67} among various other disciplines. For a general introduction to POD and reduced order modeling, we refer to e.g.\ \cite{graessle21, gubisch17, kunisch01, pinnau08} and suggest \cite{banholzer24, kunisch02} as further reading. In order to reduce a dynamical system utilizing POD, one uses information (e.g.\ snapshots of the system's state over time) in order to compute a set of POD modes. Assuming that the solution trajectory is approximated well by the reduced space spanned by those POD basis vectors, a reduced model is constructed using a POD Galerkin ansatz and a projection onto the reduced space. Then, solving the surrogate model reduces computational time in comparison to the original high-dimensional problem. 
    
    In standard POD, the explicit time dependence plays a central role and can limit its applicability in certain cases. Typically, POD is applied for time-dependent PDEs by discretizing the spatial domain, resulting in a high-dimensional system of ordinary differential equations in time. However, some PDEs do not allow the derivation of such a system which can be solved with a time stepping procedure. Differential equations, which e.g.\ involve both initial and terminal conditions result in numerical schemes for which the application of standard POD becomes insufficient or is not applicable anymore. 

    Possible approaches to MOR simultaneously for space and time have been
    proposed such as proper generalized decompositions (PGD), space-time reduced
    bases, and space-time POD. Roughly spoken, PGD treats time, space and
    possibly parameters as individual dimensions of the problem and iteratively
    identifies low-dimensional bases, see e.g.\ \cite{ChiALK11}. Contrarily,
    the general method of reduced bases applies well to monolithic
    space-time formulations; cf.\ \cite{SteU12} for an early work and
    \cite{TenMD24} for recent results. An approach to a POD Galerkin approximation in space and time proposed in \cite{weiland2006} leads to a Sylvester equation for the expansion coefficients. For reference, we further mention the general approach
    of first applying a time-discretization scheme and consider the full but
    discretized time-dimension jointly with the space using e.g.,
    least-squares POD; see \cite{choi19}. For further developments which extend the
    scope of space-time model order reduction we refer to\ \cite{ballarin22, choi19, frame23, jin21, li26, schmidt19, yukiko21}.

    In this work, we consider space-time POD as it has originally been introduced in \cite{baumann15-poster, baumann18, baumann15} and motivated by the singular value decomposition of tensors \cite{DeDV00}. 
    In comparison to standard POD, space-time POD formalizes the projection onto
    a reduced space in standard POD and allows the projection of both time and
    space directions. In the original space-time POD works, the method has
    successfully been derived and implemented for prominent examples. The
    comparison between standard POD and space-time POD indicates, that
    space-time does not necessarily outperform standard POD, as long as
    time-stepping or discontinuous Galerkin schemes can be applied in the
    standard POD approach. This may change during the reduction of problems in
    which time is a global variable; consider, e.g.,\ finite-time control problems. 
    
    We make a contribution to the space-time POD framework by an a-priori
    analysis of the error which occurs in space-time POD between the solutions
    to the full order model (FOM) and the reduced order model (ROM). While the
    a-priori error estimation for space-time POD uses new methods rather than similar error estimations for standard POD, it is nevertheless deeply motivated by those works, above all \cite{banholzer24}. Let us also mention \cite{hinze05, kunisch01, kunisch02,singler2014} as further works about POD error estimations for an extended overview. 
    In our case, the FOM will be the Galerkin approximation of a linear parabolic evolution problem. In Theorem~\ref{thm:final_estimation}, we derive an a-priori error bound for the error
    \begin{equation}\label{eq:introduction_error_splitting}
            \Vert x - \hat x\Vert \leq \Vert x - \Pi x\Vert + \Vert \Pi x - \hat x\Vert =: \Vert \varrho\Vert + \Vert \vartheta\Vert,
    \end{equation}
    where $x$ denotes a solution to the FOM, $\Pi x$ the projection of the full order solution $x$ onto the reduced space and $\hat x$ a solution to the ROM. Splitting the error like this is a common technique, see e.g.\  \cite{hinze05, kunisch01, kunisch02, thomee06}, and allows to estimate $\Vert \varrho\Vert$  and $\Vert\vartheta\Vert$ individually. The $\varrho$-term measures the error occurring due to the projection onto a subspace and can be estimated using information about the subspace only. The $\vartheta$-term on the other hand, incorporates the error between the full order solution projected onto the reduced space and the reduced solution and thus depends on problem-specific information. 

    In this work, we focus on linear parabolic evolution problems and plan to investigate error analysis for other problems in future work. 

    We introduce the problem in the functional analytic context and consider a
    Gelfand triple $V\hookrightarrow H \cong H^* \hookrightarrow V^*$, where $V,
    H$ denote real separable Hilbert spaces. Let the dual pairing between $V^*$
    and $V$ be given as $\braket{u^*}{v}_{V^*, V} := u^*(v)$ for $v\in V$ and
    $u^*\in V^*$
defined in the dense and continuous embedding $V\hookrightarrow H$ as the
extension of the dual product in $H$, i.e. $\braket{u^*}{v}_{V^*, V} =
\braket{u^*}{v}_{H^*, H}$ if $u^* \in H^*$ and with $v\in V$ considered as $v\in
H$. Moreover, since we identify $H^*$ with $H$, it holds for the inner product in $H$ that $\braket{u}{v}_H =\braket{u}{v}_{H^*, H}$ for all $u,v\in H = H^*$. For $0 < T < \infty$, $f\in L^2(0,T; V^*)$ and a second order differential operator $\A(t)\in\L(V, V^*)$ we consider the evolution problem
\begin{equation}\label{eq:evolution_problem}
    \left\{
    \begin{array}{rll}
         x_t(t) + \A(t)x(t) &= f(t) \quad & \text{in }V^*, \text{ for a.e. }t\in (0,T),\\
         x(0) &= x_0 & \text{in }H.
    \end{array}
    \right.
\end{equation}
The operator $\A(t)$ is given via $\langle \A(t)\varphi,\psi\rangle_{V^*, V} :=
a(t;\varphi, \psi)$ for $\varphi,\psi\in V$ for a bilinear form $a(t;\cdot,
\cdot):V\times V\to \R$, such that $t\mapsto a(t; \varphi, \psi)$ is measurable
on $[0,T]$ for all $\varphi,\psi\in V$. Furthermore, we assume that the bilinear
form $a(t;\cdot, \cdot)$ is continuous and coercive uniformly with respect to $t$, i.e.\ there are constants $0 < \gamma <
\infty$ and $0 <\alpha < \infty$, such that
\begin{subequations}\label{eq:a_coercive_continuous}
  \begin{align}
    |a (t;\varphi, \psi)| &\leq \gamma \Vert \varphi\Vert_V\Vert \psi\Vert_V\quad &&\forall \varphi, \psi \in V,   \label{eq:a_coercive_continuous:a}\\
   a(t; \varphi, \varphi) &\geq \alpha \Vert \varphi \Vert_V^2\quad &&\forall \varphi \in V.     \label{eq:a_coercive_continuous:b}
  \end{align}
\end{subequations}
 %   \begin{equation}\label{eq:a_coercive_continuous}
 %      \begin{array}{rll}
 %           |a (t;\varphi, \psi)| &\leq \gamma \Vert \varphi\Vert_V\Vert \psi\Vert_V\quad &\forall \varphi, \psi \in V,\\
 %           a(t; \varphi, \varphi) &\geq \alpha \Vert \varphi \Vert_V^2\quad &\forall \varphi \in V.
 %       \end{array}
 %   \end{equation}
The evolution problem \eqref{eq:evolution_problem} can also be stated in its variational form: find $x \in W(0,T;V) := \{v \in L^2(0,T;V)\mid v_t\in L^2(0,T; V^*)\}$, such that
\begin{equation}\label{eq:evolution_problem_variational}
     \left\{
    \begin{array}{rll}
         \langle x_t(t), \varphi\rangle_{V^*, V} + a(t; x(t), \varphi) &= \langle f(t), \varphi\rangle_{V^*, V} \quad & \forall \varphi \in V, \text{ a.e. }t\in (0,T),\\
         x(0) &= x_0,  & \text{in }H.
    \end{array}
    \right.
\end{equation}
    
    \subsection{Paper organization}
    Since this work is deeply rooted in the context of space-time POD, we introduce the fundamental concepts and cite the main results for space-time POD which are needed for this work in Section~\ref{sec:space_time_pod}. Furthermore, we introduce the related notation for space-time POD and define the full order model as well as the reduced order model in this section.

    Section~\ref{sec:approximation_error} contains the analysis of the error between the FOM and the ROM $\Vert x - \hat x\Vert$, as well as the derivation of an a-priori error bound for this error. This section is divided into two parts, where each subsection is concerned with one of the terms in the splitting of the error \eqref{eq:introduction_error_splitting}. We start with the estimation of the $\varrho$-error term in Subsection~\ref{sec:rho-error}, continue with the estimation of the $\vartheta$-term in Subsection~\ref{sec:theta-error} and finally combine those estimates into an overall error bound in Subsection~\ref{sec:error-summary}.
    
    In Section~\ref{sec:numerical examples} we will execute the numerical analysis of our theoretical findings. We start with the application of space-time POD for given reduced dimensions and continue with the investigation of the error terms in dependence on varying choices for the reduced dimensions. Since the main focus of this section is the investigation of the error terms, we refer to \cite[Sections~3,4,5]{baumann18} for the technical implementation of space-time POD. 

\section{Application of space-time POD}\label{sec:space_time_pod}
In this section, we discuss the application of space-time POD to the parabolic evolution problem \eqref{eq:evolution_problem}. For this, we follow closely the notation and cite results from \cite{baumann18}, where space-time POD has been introduced. 

In Subsection~\ref{sec:space_time_pod_spaces_notation}, we introduce the notation and collect the results from space-time POD. This especially includes the product space structure on which space-time POD is built upon, as well as how to construct the reduced spaces. After the introduction of the product spaces, we introduce the full order model (FOM) and the reduced order model (ROM) in Subsection~\ref{sec:space_time_pod_fom_rom}. Since the initial condition requires special treatment in space-time POD, we devote Subsection~\ref{sec:space_time_pod_initial} to the discussion of the initial condition.

\subsection{Notation and product space structure}\label{sec:space_time_pod_spaces_notation}
The space-time POD framework is built in the context of finite-dimensional function spaces with a (tensor) product structure. We follow the notation in \cite{baumann18} and introduce the $s$-dimensional space $\S$ and $q$-dimensional space $\Y$ as
$$\S = \text{span}\{\psi_1,\ldots,\psi_s\}\subset L^2(0,T) \quad\text{and}\quad \Y = \text{span}\{\nu_1,\ldots,\nu_q\}\subset L^2(\Omega).$$
The $s\cdot q$-dimensional (tensor) product space $S\cdot \Y \subset L^2(0,T;L^2(\Omega))$ is the space spanned by the basis $\{\psi_j\cdot \nu_i\}^{j=1,\ldots,s}_{i=1,\ldots,q}$, such that each $x\in \S\cdot \Y$ can be written as
$$x(\tau, \xi) = \sum_{j=1}^s\sum_{i=1}^q \mathbf{x}_{i,j}\nu_i(\xi)\psi_j(\tau),$$
where $\mathbf{x}_{i,j}$ are entries of the coefficient matrix $\mathbf{X} = [\mathbf{x}_{i,j}]_{i=1,\ldots,q}^{j=1,\ldots,s} \in \mathbb{R}^{q \times s}$. We denote the columnwise vectorization of the coefficient matrix as
$$\mathbf{x} = [\mathbf{x}_{1,1},\ldots,\mathbf{x}_{q,1}, \mathbf{x}_{1,2},\ldots, \mathbf{x}_{q,2},\ldots, \mathbf{x}_{1,s},\ldots, \mathbf{x}_{q,s}]^T =: \mathrm{vec}(\mathbf{X}) \in \mathbb{R}^{qs}.$$
We introduce the $L^2$-inner products on $\S$ and $\Y$, respectively, as
$$(s_1,s_2)_{\S} = \int_0^Ts_1s_2\intd \tau\quad\text{and}\quad (y_1,y_2)_\Y = \int_\Omega y_1y_2\intd \xi.$$
The Gramians with respect to the basis in space $\{\nu_i\}_{i=1}^q$ and the basis in time $\{\psi_j\}_{j=1}^s$ can be introduced as 
$$ \M_\S := [(\psi_i, \psi_j)_\S]_{i=1,\ldots,s}^{j=1,\ldots,s}, \qquad \M_\Y := [(\nu_i, \nu_j)_\Y]_{i=1,\ldots,q}^{j=1,\ldots,q}$$
and we denote their Cholesky factorizations as $\M_\S = \mathbf{L}_\S\mathbf{L}_\S^T$ and $\M_\Y = \mathbf{L}_\Y\mathbf{L}_\Y^T$. In case of $\S\subseteq H^1(0,T)$, we define the matrix
$$\K_\S := [(\dot\psi_i, \dot\psi_j)_\S]_{i=1,\ldots,s}^{j=1,\ldots,s}$$
with a factorization $\K_\S = \mathbf{J}_\S\mathbf{J}_\S^T$. In case of $\Y \subseteq H^1(\Omega)$, we define
$$\K_\Y := [(\nabla\nu_i, \nabla\nu_j)_\Y]_{i=1,\ldots,q}^{j=1,\ldots,q}$$
with its factorization $\K_\Y = \mathbf{J}_\Y\mathbf{J}_\Y^T$. An inner product on $\S\cdot \Y$ is given as
$$(x_1,x_2)_{\S \cdot \Y} := \int_0^T\int_\Omega x_1  x_2\intd \xi \intd \tau$$
with the corresponding norm $\Vert x \Vert^2_{\S\cdot\Y} := (x,x)_{\S \cdot \Y}$. We can represent the norm in $\S\cdot \Y$ with respect to the Frobenius norm. The following lemma extends the results in \cite[Corollary~2.4]{baumann18}.\\

\begin{lemma}[Representation of $\S\cdot \Y$ norm]\label{lem:rewrite_x}
	Assume, that $\S\cdot \Y$ provides the necessary regularity, i.e.\ $\S
  \subseteq H^1(0,T)$ if the derivative in time occurs, as well as $\Y \subseteq
  H^1(\Omega)$ if the derivative in space occurs. Consider the factorizations of the Gramians $\M_\S = \mathbf L_\S\mathbf L_\S^T$, $\M_\Y = \mathbf L_\Y\mathbf L_\Y^T$, $\K_\S = \mathbf{J}_\S\mathbf{J}_\S^T$ and $\K_\Y = \mathbf{J}_\Y\mathbf{J}_\Y^T$. Then, for a given $x\in \S\cdot \Y$ with its coefficient matrix $\bf{X}$ and vector $\mathbf{x} = \mathrm{vec}(\mathbf{X})$ it holds that
	\begin{equation*}
		\begin{split}
			 \Vert x \Vert_{\S \cdot \Y}^2 &= \Vert \mathbf{x}\Vert^2_{\M_\S \otimes \M_\Y} = \Vert \mathbf{L}_\Y^T\mathbf{X}\mathbf{L}_\S\Vert_F^2,\\
			\Vert x_t \Vert_{\S \cdot \Y}^2 &= \Vert \mathbf{x}\Vert^2_{\K_\S \otimes \M_\Y} = \Vert \mathbf{L}_\Y^T\mathbf{X}\mathbf{J}_\S\Vert_F^2,\\
			\Vert \nabla x \Vert_{\S \cdot \Y}^2 &= \Vert \mathbf{x}\Vert^2_{\M_\S \otimes \K_\Y} = \Vert \mathbf{J}_\Y^T\mathbf{X}\mathbf{L}_\S\Vert_F^2.
		\end{split}
	\end{equation*}
\end{lemma}
\begin{proof}
    The proof is based on the proof of \cite[Corollary~2.4]{baumann18}, where the formula for $\Vert x \Vert^2_{\S\cdot \Y}$ has been shown. Consider the differential operator or the identity $D \in\{\text{id}, \partial_t, \nabla\}$ and let
    $$\mathbf \Gamma_\S = \begin{cases}
    	\M_\S & \text{if } D \in \{\text{id}, \nabla\},\\
	\K_\S & \text{if } D =\partial_t,
    \end{cases} \qquad 
    \mathbf \Gamma_\Y = \begin{cases}
    	\M_\Y & \text{if } D \in \{\text{id}, \partial_t\},\\
	\K_\Y & \text{if } D =\nabla,
    \end{cases}$$
    as well as
    $$
    \mathbf \Lambda_\S = \begin{cases}
    	\mathbf L_\S & \text{if } D \in \{\text{id}, \nabla\},\\
	\mathbf J_\S & \text{if } D = \partial_t,
    \end{cases}, \qquad 
    \mathbf \Lambda_\Y = \begin{cases}
    	\mathbf L_\Y & \text{if } D \in \{\text{id}, \partial_t\},\\
	\mathbf J_\Y & \text{if } D = \nabla.
    \end{cases}
    $$
    Then, simple computations show that
    $$\Vert D x \Vert^2_{\S\cdot\Y} = \int_0^T\int_\Omega (D x)^2\intd x\intd t = \mathbf{x}^T (\mathbf \Gamma_\S \otimes \mathbf \Gamma_\Y)\mathbf{x} = \Vert \mathbf{x}\Vert^2_{\mathbf \Gamma_\S \otimes \mathbf \Gamma_\Y}.$$
    Furthermore, it follows from basic properties and relations between the Kronecker product, the vectorization operator and the Frobenius norm that
    \begin{equation*}
        \begin{split}
            \Vert \mathbf{x}\Vert^2_{\mathbf \Gamma_\S \otimes \mathbf \Gamma_\Y} &= \mathbf{x}^T (\mathbf \Gamma_\S \otimes \mathbf \Gamma_\Y)\mathbf{x} = \mathbf{x}^T (\mathbf \Lambda_\S \otimes \mathbf \Lambda_\Y)(\mathbf \Lambda_\S^T \otimes \mathbf \Lambda_\Y^T)\mathbf{x}\\
            &= \Vert (\mathbf \Lambda_\S^T \otimes \mathbf \Lambda_\Y^T)\mathbf{x}\Vert^2_2 = \Vert\mathrm{vec}(\mathbf \Lambda_\Y^T \mathbf{X}\mathbf \Lambda_\S)\Vert_2^2 = \Vert \mathbf \Lambda_\Y^T \mathbf{X}\mathbf \Lambda_\S\Vert_F^2.
        \end{split}
    \end{equation*}
\end{proof}

Now, let $\hat \S \subseteq \S$ be an $\hat s$-dimensional subspace which is spanned by the reduced basis $\{\hat \psi_j\}_{j=1,\ldots,\hat s}$ and $\hat \Y \subseteq \Y$ be a $\hat q$-dimensional subspace, which is spanned by the reduced basis $\{\hat \nu_i\}_{i=1,\ldots,\hat q}$. Then, the reduced product space $\hat \S\cdot\hat\Y$ is spanned by the basis $\{\hat\psi_j \cdot \hat \nu_i\}_{i=1,\ldots,\hat q}^{j=1,\ldots,\hat s}$. We use the following two Lemmas from \cite{baumann18} to determine the reduced basis functions $\{\hat \nu_i\}_{i=1}^{\hat q}$ and $\{\hat \psi_j\}_{j=1}^{\hat s}$ for the reduced space $\hat \S \cdot \hat \Y$ from given bases $\{ \nu_i\}_{i=1}^{ q} \subset \Y$ and $\{ \psi_j\}_{j=1}^{ s}\subset \S$ in an optimal way with respect to the $\| \cdot \|_{\S \cdot \Y}$-norm. The reduced basis in space $\{\nu_i\}_{i=1}^q$ can be computed using \cite[Lemma 2.5]{baumann18}. The lemma reads as:\\

\begin{lemma}[Optimal low-rank bases in space {\cite[Lemma 2.5]{baumann18}}]\label{lem:space-basis}
    Given $x\in \S\cdot \Y$ and the associated matrix of coefficients $\mathbf{X}$. The best-approximating $\hat q$-dimensional subspace $\hat \Y$ in the sense that the projection error $\Vert x - \Pi_{\S\cdot \hat \Y}x \Vert_{\S\cdot \Y}$ is minimal over all subspaces of $\Y$ of dimension $\hat q$ is given as $\text{span}\{\hat \nu_i\}_{i=1,\ldots,\hat q}$, where 
    \begin{equation}\label{eq:space_equation}
        \begin{bmatrix}
            \hat \nu_1\\\hat\nu_2\\\vdots\\\hat\nu_{\hat q}
        \end{bmatrix}
         = V_{\hat q}^T \mathbf{L}_{\Y}^{-1}
         \begin{bmatrix}
             \nu_1\\
             \nu_2\\
             \vdots\\
             \nu_q
         \end{bmatrix},
    \end{equation}
    where $V_{\hat q}$ is the matrix of the $\hat q$ leading left singular vectors of 
    $$\mathbf{L}_\Y^T\mathbf{X}\mathbf{L}_\S.$$
\end{lemma}

The proof can be found in \cite{baumann18}. It treats the reduction in space direction. A reduced basis in time direction can be found by virtue of \cite[Lemma 2.7]{baumann18}, which we quote as:\\

\begin{lemma}[Optimal low-rank bases in time {\cite[Lemma 2.7]{baumann18}}]\label{lem:time-basis}
    Given $x\in \S\cdot \Y$ and the associated matrix of coefficients $\mathbf{X}$. The best-approximating $\hat s$-dimensional subspace $\hat \S$ in the sense that the projection error $\Vert x - \Pi_{\hat S\cdot \Y}x \Vert_{\S\cdot \Y}$ is minimal over all subspaces of $\S$ of dimension $\hat s$ is given as $\text{span}\{\hat \psi_j\}_{j=1,\ldots,\hat s}$, where 
    \begin{equation}\label{eq:time_equation}
        \begin{bmatrix}
            \hat \psi_1\\\hat\psi_2\\\vdots\\\hat\psi_{\hat s}
        \end{bmatrix}
         = U_{\hat s}^T \mathbf{L}_{\S}^{-1}
         \begin{bmatrix}
             \psi_1\\
             \psi_2\\
             \vdots\\
             \psi_s
         \end{bmatrix},
    \end{equation}
    where $U_{\hat s}$ is the matrix of the $\hat s$ leading right singular vectors of
    $$\mathbf{L}_\Y^T\mathbf{X}\mathbf{L}_\S.$$
\end{lemma}

The consideration of a projected solution $\hat x\in \hat \S \cdot \hat \Y$ requires the consecutive application of the two projections described by Lemma~\ref{lem:space-basis} and Lemma~\ref{lem:time-basis}. We denote
\begin{equation}\label{eq:two_projections}
    \mathcal P_{\hat \Y\to \hat \S} = \Pi_{\hat \S \cdot \hat \Y}\circ \Pi_{\S \cdot \hat \Y}\quad\text{and}\quad\mathcal P_{\hat \S\to \hat \Y} = \Pi_{\hat \S \cdot \hat \Y}\circ \Pi_{\hat\S \cdot \Y}.
\end{equation}
Hence, $\mathcal P_{\hat \Y\to \hat \S}$ denotes the projection of a function onto $\S\cdot\hat \Y$, followed by the projection onto $\hat \S\cdot\hat \Y$, while $\mathcal P_{\hat \S\to \hat \Y}$ denotes the projection onto $\hat \S \cdot \Y$ followed by the projection onto $\hat \S\cdot\hat \Y$. If we aim to project onto $\hat \S \cdot \hat \Y$, we have to choose between using $\mathcal P_{\hat \Y\to \hat \S}$ or $\mathcal P_{\hat \S\to \hat \Y}$.

\subsection{Full order model and reduced order model}\label{sec:space_time_pod_fom_rom}
Let us consider a space-time variational formulation of the evolution problem \eqref{eq:evolution_problem} and introduce the bilinear form $B:W(0,T;V)\times L^2(0,T; V)\to\R$ and linear form $L: L^2(0,T; V) \to \R$ as
\begin{equation}\label{eq:bilinear_linear_form}
    \begin{split}
        B(x, \varphi) &:= \int_0^T \langle x_t(t), \varphi(t)\rangle_{V^*,V} + a(t, x(t), \varphi(t))\intd t,\\
        L(\varphi) & := \int_0^T \langle f(t), \varphi(t)\rangle_{V^*, V}\intd t.
    \end{split}
\end{equation}
The problem in space-time is then given as: find $x \in W(0,T;V)$, such that
\begin{equation}\label{eq:space_time_variational}
    \left\{
    \begin{array}{rll}
        B(x, \varphi) &= L(\varphi),\quad &\forall \varphi\in L^2(0,T; V),\\
        x(0) &= x_0, & \text{in }H.
    \end{array}
    \right.
\end{equation} 

The space-time variational formulation \eqref{eq:space_time_variational} is equivalent to the variational formulation \eqref{eq:evolution_problem_variational} as it has been shown in \cite[Theorem~1.33]{hinze-ulbrich}, i.e.\ we have the following.\\

\begin{theorem}[{Equivalence of formulations, \cite[Theorem~1.33]{hinze-ulbrich}}]
    The variational formulations \eqref{eq:evolution_problem_variational} and \eqref{eq:space_time_variational} are equivalent.\\
\end{theorem}

For $s$-dimensional $\S \subset H^1(0,T)$ and $q$-dimensional $\Y \subset H^1_0(\Omega)$, we consider the Galerkin approximation of \eqref{eq:space_time_variational} onto the product space $\S\cdot \Y$. We say that $x\in\S\cdot\Y \subset W(0,T;H^1_0(\Omega))$ is a solution to the FOM, if it satisfies
\begin{equation}\label{eq:fom}
    \left\{
    \begin{split}
         B(x, \nu_i\psi_j) &= L(\nu_i\psi_j),\\
         x(0)&=  \Pi_\Y x_0
    \end{split}
    \right.
\end{equation}
for all $j=1,\ldots,s$, $i=1,\ldots,q$ and $x_0 \in H^1_0(\Omega)$. For $\hat s \leq s$ and $\hat q \leq q$, we consider the $\hat s$-dimensional subspace $\hat \S \subseteq \S$ and the $\hat q$-dimensional subspace $\hat \Y \subset \hat \Y$. We call $\hat x \in \hat \S \cdot \hat \Y$ a solution to the ROM, if it satisfies
\begin{equation}\label{eq:rom}
    \left\{
    \begin{split}
         B(\hat x, \hat \nu_i\hat \psi_j) &= L(\hat \nu_i\hat \psi_j),\\
         \hat x(0)&= \Pi_{\hat \Y} x_0
    \end{split}
    \right.
\end{equation}
for all $i=1,\ldots,\hat q$ and $j=1,\ldots,\hat s$. 

\subsection{Incorporation of the initial condition}\label{sec:space_time_pod_initial}
In the proposed error analysis for the reduced order solution $\hat x$, we will
rely on the assumption that $\hat x(0) = \Pi_\Y x_0$, i.e., that the initial condition has not been reduced but only projected to the truth finite element space $\Y$. Because of linearity of the overall problem, a
nonzero initial error can be added by superposition arguments. However,
this will increase the complexity of the analysis without providing additional practical and theoretical advantages. Moreover, as we will lay out below, a nonzero initial error can also be treated by straight-forward adaptation of the bases; cf.\ the discussion in \cite[Section~4.2]{baumann18}.

If we choose $\{\hat \psi_j\}_{j=1}^{\hat s}$ during the projection of $\S$ onto $\hat\S$ as described in Section~\ref{sec:space_time_pod_spaces_notation}, i.e.\ as $\hat \S = \text{span}\{\hat \psi_1,\ldots,\hat \psi_{\hat s}\}$ based on the full order basis $\{\psi_j\}_{j=1}^{s}$ as computed in Lemma~\ref{lem:time-basis}, we don't have any a-priori knowledge about the basis functions $\{\hat \psi_j\}_{j=1}^{\hat s}$. Hence, we cannot guarantee that the initial condition $\hat x(0) = \Pi_\Y x_0$ will (and even can) be satisfied. The following steps describe how the initial condition can be implemented instead:
\begin{itemize}
    \item[1.] We choose $\S = \text{span}\{\psi_1,\ldots,\psi_s\} \subset H^1(0,T)$, where $\{\psi_j\}_{j=1}^s$ is a nodal basis, such that $\psi_1$ is the node associated with the function values for $t=0$.

    \item[2.] Set the first column of the matrix of coefficients in $\mathbf{X}$ to zero in order to get the matrix $\accentset{\circ}{\mathbf{X}}$. We use this matrix in order to compute $\accentset{\circ}{U}_{\hat s-1}$ as the matrix of the $\hat s-1$ leading right singular vectors of $\mathbf{L}_\Y^T\accentset{\circ}{\mathbf{X}}\mathbf{L}_\S$.

    \item[3.] We set 
    \begin{equation}\label{eq:different_U}
    U_{\hat s} = \begin{bmatrix}
        \mid & \\
        (\mathbf{L}_\S^T)_{:,1} & \accentset{\circ}{U}_{\hat s-1}\\
        \mid & \\
    \end{bmatrix},
    \end{equation}
    where $(\mathbf{L}_\S^T)_{:,1}$ denotes the first column of $\mathbf{L}_\S^T$. Then, we compute the reduced basis as in Lemma~\ref{lem:time-basis} as 
    \begin{equation}\label{eq:set_red_basis_different}
        \begin{bmatrix}
            \hat \psi_1\\\hat\psi_2\\\vdots\\\hat\psi_{\hat s}
        \end{bmatrix}
         = U_{\hat s}^T \mathbf{L}_{\S}^{-1}
         \begin{bmatrix}
             \psi_1\\
             \psi_2\\
             \vdots\\
             \psi_s
         \end{bmatrix},
    \end{equation}
    
\end{itemize}
    Under the assumption that $\mathbf L_\S$ is an upper-triangular factor, this construction ensures that $\hat \psi_1 = \psi_1$ and that only $\hat \psi_1$ is associated with the value $\hat x(0)$ since it holds that $\hat \psi_2(0) = \cdots=\hat\psi_{\hat s}(0) = 0$, while the trajectory for $t>0$ still gets optimally approximated by $\hat \psi_2,\ldots,\hat \psi_{\hat s}$. Notice that the choice of the first column of $U_{\hat s}$ as the first column of $\mathbf{L}_\S^T$ differs from the procedure in \cite[Section~4.2]{baumann18}. This is required in order to get $\hat \psi_1 = \psi_1$. We denote
    $$\mathscr{S}\cdot \Y = \left\{ x \in \S\cdot \Y\;\middle|\;  x(0) = \Pi_\Y x_0\right\}$$
    as well as
    $$\hat {\mathscr{S}}\cdot \Y = \left\{\hat x \in \hat{\S}\cdot \Y\;\middle|\; \hat x(0) = \Pi_\Y x_0\right\},$$
    where we have computed the basis by which $\hat\S$ is spanned as described above. Notice that this construction implies for $x\in \mathscr{S}\cdot \Y$ and $\hat x \in \hat{\mathscr{S}}\cdot \Y$, that 
    \begin{equation}\label{eq:teta_zero_zero}
    x(0) - \hat{x}(0) = \Pi_{\Y} x_0 - \Pi_{\Y} x_0 = 0
    \end{equation}

    \begin{remark}[Only time reduction in this section]\label{rem:simple-notation}
        In this section, we have introduced the initial condition with respect to the space $\S\cdot \Y$. Hence, this is the case in which the time dimension is treated before the space reduction. The treatment of the initial condition works exactly the same way when we project the space dimension first, i.e.\ for $\S\cdot \hat \Y$. Especially the sets $\mathscr{S}\cdot \hat \Y$ and $\hat{\mathscr{S}}\cdot \hat \Y$ are defined similar as described in this section. We define the projections $\mathcal P_{\hat \Y\to\hat{\mathscr{S}}}$ and $\mathcal P_{\hat{\mathscr{S}}\to\hat\Y}$ similar to the definition of $\mathcal P_{\hat \Y\to\hat{\S}}$ and $\mathcal P_{\hat{\S}\to\hat\Y}$ in \eqref{eq:two_projections}.\\
    \end{remark}

    \begin{remark}[Orthogonality]\label{rem:not_orthogonal}
        Notice, that a reduced time basis, which has been computed directly with Lemma~\ref{lem:time-basis} is orthogonal, see \cite[Remark~2.6]{baumann18}. In contrast, the reduced basis $\{\hat \psi_j\}_{j=1}^{\hat s}$ constructed by \eqref{eq:set_red_basis_different} is not orthogonal anymore since
        \begin{equation}\label{eq:red_basis_not_ortho}
            \begin{split}
                \M_{\hat{\mathscr{S}}} &= \int_0^T \begin{bmatrix}
                \hat \psi_1(\tau)\\\vdots\\\hat \psi_{\hat s}(\tau)
            \end{bmatrix}\begin{bmatrix}
                \hat \psi_1(\tau)&\ldots&\hat \psi_{\hat s}(\tau)
            \end{bmatrix}\intd \tau \\
            &= U_{\hat s}^T\mathbf{L}_{\S}^{-1} \int_0^T \begin{bmatrix}
                 \psi_1(\tau)\\\vdots\\ \psi_{\hat s}(\tau)
            \end{bmatrix}\begin{bmatrix}
                \psi_1(\tau)&\ldots&\psi_{\hat s}(\tau)
            \end{bmatrix}\intd \tau \mathbf{L}_\S^{-T}U_{\hat s} 
            = U_{\hat s}^T\mathbf{L}_{\S}^{-1}\M_\S \mathbf{L}_\S^{-T}U_{\hat s}\\
            &= U_{\hat s}^T U_{\hat s}  \neq I_{\hat s},
            \end{split}
        \end{equation}
         where $(\mathbf{L}_\S)_{1,:}$ denotes the first row of $\mathbf{L}_\S$ and $I_{\hat s}$ the $\hat s$-dimensional identity. The missing orthogonality must be minded during the derivation of the error bounds in Section~\ref{sec:rho-error} as well as the computation of the reduced solution during the numerical examples in Section~\ref{sec:numerical examples}.
         %\left[\begin{matrix}
         %       (\mathbf{L}_\S)_{1,:}(\mathbf{L}_\S^T)_{:,1}& \begin{matrix}
         %           (\mathbf{L}_\S)_{1,:}\accentset{\circ}{u}_1 &  \cdots & (\mathbf{L}_\S)_{1,:}\accentset{\circ}{u}_{\hat s - 1}
         %       \end{matrix}\\
         %       \begin{matrix}
         %       \accentset{\circ}{u}_1^T(\mathbf{L}_\S^T)_{:,1} \\ \vdots\\ \accentset{\circ}{u}_{\hat s - 1}^T(\mathbf{L}_\S^T)_{:,1}
         %       \end{matrix}&
         %       \tikz[anchor=base, baseline]{\node[draw, minimum width=3.5cm, minimum height=1.3cm] {$I_{\hat s - 1}$}}
         %   \end{matrix}\right]
    \end{remark}

\section{Analysis of the approximation error}\label{sec:approximation_error}
In this section, we estimate the approximation error by
\begin{equation}\label{eq:error}
    \Vert x - \hat x\Vert_{\S\cdot \Y} \leq \Vert x - \Proj x \Vert_{ \S \cdot  \Y} + \Vert \Proj x- \hat x\Vert_{\S\cdot \Y} =: \Vert \varrho \Vert_{\S\cdot \Y} +\Vert \vartheta \Vert_{\S\cdot \Y},
\end{equation}
where $x\in \mathscr{S}\cdot \Y$ is a solution to the full order problem \eqref{eq:fom} and $\hat x\in \hat{\mathscr{S}} \cdot \hat \Y$ is a solution to the reduced problem \eqref{eq:rom}. The projection $\Proj:\mathscr{S}\cdot\Y \to \hat{\mathscr{S}}\cdot \hat \Y$ denotes either $\Proj_{\hat \Y \to \hat{\mathscr{S}}}$ or $\Proj_{\hat{\mathscr{S}}\to \hat \Y}$ (cf. \eqref{eq:two_projections} or Remark~\ref{rem:simple-notation}). In \eqref{eq:error}, we have split the error using the notions
\begin{equation}\label{eq:theta_rho}
    \varrho = x - \Proj x,\qquad
        \vartheta = \Proj x - \hat x.
\end{equation}
While $\Vert \varrho\Vert_{\S\cdot \Y}$ measures the error between the full order solution \eqref{eq:fom} in $\mathscr{S}\cdot \Y$ and its projection onto the reduced space $\hat{\mathscr{S}}\cdot\hat \Y$, the term $\Vert \vartheta\Vert_{\S\cdot \Y}$ indicates how much the projection onto $\hat{\mathscr{S}}\cdot\hat \Y$ differs from a solution of the reduced order model \eqref{eq:rom}. Such splitting of the error allows to estimate the $\varrho$-error and the $\vartheta$-error separately from each other. The $\varrho$-error is only dependent on the projection onto $\hat{\mathscr{S}}\cdot\hat \Y$ and not on the full order model \eqref{eq:fom} or the reduced order model \eqref{eq:rom}. Its estimation will be carried out in Section~\ref{sec:rho-error}. In comparison, the estimation of the $\vartheta$-error requires knowledge about the bilinear and linear form in \eqref{eq:fom} and \eqref{eq:rom} and is specific for each problem. We will present its estimation for the evolution problem \eqref{eq:evolution_problem} in Section~\ref{sec:theta-error}. In Section~\ref{sec:error-summary}, we combine the results from both sections to estimate the entire error.

\subsection{Error bound of $\varrho$-error}\label{sec:rho-error}
In this section, we derive a bound for the error term $\Vert\varrho\Vert_{\S\cdot\Y}$ from \eqref{eq:error}. This will be achieved in Proposition~\ref{prop:rho_error}. Furthermore, we derive similar bounds for $\Vert \varrho_t\Vert_{\S\cdot\Y}$ and $\Vert \nabla \varrho\Vert_{\S\cdot\Y}$ under appropriate regularity assumptions. Both error bounds are governed by the singular value based notions

\begin{equation}\label{eq:singular_value_bounds}
    \Sigma_{\hat\Y \rightarrow \hat\S} = \left( \sum_{i=\hat q + 1}^q \sigma_i^2\right)^{\frac{1}{2}} + \left(\sum_{j=\hat s}^s \accentset{\circ}{\sigma}_j^2 \right)^{\frac{1}{2}},\quad
    \Sigma_{\hat\S \rightarrow \hat\Y} = \left(\sum_{j=\hat s}^s \underline{\accentset{\circ}{\sigma}}_j^2\right)^{\frac{1}{2}} + \left(\sum_{i=\hat q + 1}^q \smash{\underline \sigma}_i^2\right)^{\frac{1}{2}},
\end{equation}
where $\{\sigma_j\}_{j=1}^q$ denote the singular values of $\mathbf{L}_\Y^T\mathbf{X}\mathbf{L}_\S$, $\{\accentset{\circ}{\sigma}_j\}_{j=1}^s$ the singular values of $\mathbf{L}_{ \Y}^T\accentset{\circ}{\mathbf{X}}_{\mathscr{S}\cdot \hat \Y}\mathbf{L}_\S$, $\{\underline{\accentset{\circ}{\sigma}}_j\}_{j=1}^s$ the singular values of $\mathbf{L}_{ \Y}^T\accentset{\circ}{\mathbf{X}}\mathbf{L}_{\S}$ and $\{\smash{\underline \sigma}_i\}_{i=1}^q$ the singular values of $\mathbf{L}_{ \Y}^T\mathbf{X}_{\hat{\mathscr{S}} \cdot \Y}\mathbf{L}_{\S}$. In this, the matrices $\mathbf{X}$, $\mathbf{X}_{\mathscr{S} \cdot \hat{\Y}}$ and $\mathbf{X}_{\hat{\mathscr{S}} \cdot \Y}$ denote the coefficient matrices of $x$, $\Pi_{\mathscr{S} \cdot \hat{\Y}}x$ and $\Pi_{\hat{\mathscr{S}} \cdot \Y}x$. An overset "$\circ$" indicates the alteration of those matrices by setting the first column to zero. The following lemma shows the connection between the occurring coefficient matrices in the time-projection process.\\

\begin{lemma}\label{lem:non-orthogonal-coeff-mat-connection}
Let $\mathbf{X}_{\mathscr{S}\cdot \Y}$ and $\mathbf{X}_{\hat{\mathscr{S}}\cdot \Y}$ be the coefficient matrix of a function in $\mathscr{S}\cdot \Y$ and its projection onto $\hat{\mathscr{S}}\cdot\Y$. Furthermore, let the matrices $\accentset{\circ}{\mathbf{X}}_{\mathscr{S}\cdot \Y}$ and $\accentset{\circ}{\mathbf{X}}_{\hat{\mathscr{S}}\cdot \Y}$ result from setting the first column of $\mathbf{X}_{\mathscr{S}\cdot \Y}$ and $\mathbf{X}_{\hat{\mathscr{S}}\cdot \Y}$ to zero. Then we have that
$$
\accentset{\circ}{\mathbf{X}}_{\hat{\mathscr{S}}\cdot \Y} 
= \accentset{\circ}{\mathbf{X}}_{\mathscr{S}\cdot \Y} \mathbf{L}_\S \accentset{\circ}{U}_{\hat s - 1} \accentset{\circ}{U}_{\hat s - 1}^T\mathbf{L}_\S^{-1},$$
where the columns of $\accentset{\circ}{U}_{\hat s - 1}$  are the $\hat s - 1$ leading right singular vectors of $\mathbf{L}_\Y\accentset{\circ}{\mathbf{X}}_{\mathscr{S}\cdot \Y}\mathbf{L}_\S$.
\end{lemma}

\begin{proof}
    Notice, that this proof follows similar ideas as the proofs of Lemma~\ref{lem:space-basis} and Lemma~\ref{lem:time-basis}, with the exception, that orthogonality is only guaranteed between the reduced basis functions $\hat\psi_2,\ldots,\hat\psi_{\hat s}$ (see Remark~\ref{rem:not_orthogonal}). For the space dimension at fixed index $i$, we consider 
$$y := \sum_{j=1}^s \mathbf{x}_{i,j}\psi_j = 
\begin{bmatrix}
\mathbf{x}_{i, 1} & \ldots & \mathbf{x}_{i,s}
\end{bmatrix}
\begin{bmatrix}
    \psi_1\\\vdots\\\psi_s
\end{bmatrix}\in\mathscr{S}$$
and determine the coefficients of the orthogonal projection $\hat y = \sum_{j=1}^{\hat s} \beta_j \hat\psi_j$ of $y$ onto $\hat{\mathscr{S}}$. Therefore, we write $y$ as a function in $\hat{\mathscr{S}}$ and a remainder $\hat R$ in the orthogonal complement:
$$
y = \begin{bmatrix}
\mathbf{x}_{i,1} & \ldots & \mathbf{x}_{i,s}
\end{bmatrix}
\begin{bmatrix}
    \psi_1\\\vdots\\\psi_s
\end{bmatrix}
=
\begin{bmatrix}
\beta_1&\ldots&\beta_{\hat s}
\end{bmatrix}
\begin{bmatrix}
\hat\psi_1\\\vdots\\\hat \psi_{\hat s}
\end{bmatrix}
+ \hat R
$$
Since $\{\psi_j\}_{j=1}^{s}$ is a nodal basis and we have by construction that $\hat \psi_2(0) = \ldots = \hat \psi_{\hat s}(0) = 0$, it follows from the fact that the first column of $\accentset{\circ}{\mathbf{X}}_{\hat{\mathscr{S}}\cdot \Y}$ is zero, that $\hat y (0) = 0$ and therefore $\beta_1 = 0$. In combination with the orthogonality of the basis vectors $\hat\psi_2,\ldots,\hat\psi_{\hat s}$, the orthogonality against $\hat R$ and equations \eqref{eq:different_U} and \eqref{eq:set_red_basis_different}, it follows that
\begin{equation*}
\begin{split}
    \begin{bmatrix}
        \beta_2,\ldots,\beta_{\hat s}
    \end{bmatrix} I_{\hat s - 1} &= \int_0^T
    \underbrace{\beta_1\hat\psi_1(\tau)}_{=0}+
    \begin{bmatrix}
        \beta_2,\ldots,\beta_{\hat s}
    \end{bmatrix}
    \begin{bmatrix}
    \hat \psi_2(\tau)\\\vdots\\\hat\psi_{\hat s}(\tau)
    \end{bmatrix}
    \begin{bmatrix}
    \hat \psi_2(\tau)&\ldots&\hat\psi_{\hat s}(\tau)
    \end{bmatrix}
    \intd \tau\\
    &= \int_0^T(\hat y + \hat R) \begin{bmatrix}
    \hat \psi_2(\tau)&\ldots&\hat\psi_{\hat s}(\tau)
    \end{bmatrix}\intd \tau\\
    &= \int_0^T\begin{bmatrix}
\mathbf{x}_{i,1} & \ldots & \mathbf{x}_{i,s}
\end{bmatrix}
\begin{bmatrix}
    \psi_1\\\vdots\\\psi_s
\end{bmatrix}\begin{bmatrix}
    \hat \psi_2(\tau)&\ldots&\hat\psi_{\hat s}(\tau)
    \end{bmatrix}\intd \tau\\
    &= \begin{bmatrix}
\mathbf{x}_{i,1} & \ldots & \mathbf{x}_{i,s}
\end{bmatrix}\M_{\S}\mathbf{L}_\S^{-T}\accentset{\circ}{U}_{\hat s}\\
&= \begin{bmatrix}
\mathbf{x}_{i,1} & \ldots & \mathbf{x}_{i,s}
\end{bmatrix}\mathbf{L}_\S\accentset{\circ}{U}_{\hat s},
\end{split}
\end{equation*}
where $I_{\hat s - 1}\in \R^{\hat s -1 \times \hat s - 1}$ denotes the $\hat s -1$ dimensional identity. Using this, we can represent $\hat y$ with respect to the basis vectors of $\mathscr{S}$ as
$$\hat y = \begin{bmatrix}
    \beta_1&\ldots&\beta_{\hat s}
\end{bmatrix}\begin{bmatrix}
    \hat\psi_1\\\vdots\\\hat \psi_{\hat s}
\end{bmatrix}
=
\underbrace{\beta_1 \hat\psi_1}_{=0} + \begin{bmatrix}
    \beta_2&\ldots&\beta_{\hat s}
\end{bmatrix}\begin{bmatrix}
    \hat\psi_2\\\vdots\\\hat \psi_{\hat s}
\end{bmatrix}
=
\begin{bmatrix}
\mathbf{x}_{i,1} & \ldots & \mathbf{x}_{i,s}
\end{bmatrix}\mathbf{L}_\S\accentset{\circ}{U}_{\hat s}\accentset{\circ}{U}_{\hat s}^T \mathbf{L}_\S^{-1}\begin{bmatrix}
    \psi_1\\\vdots\\\psi_{s}
\end{bmatrix}
$$
Noting that $\begin{bmatrix}
    \mathbf{x}_{i,1}&\cdots&\mathbf{x}_{i,s}
\end{bmatrix}$ makes up for the $i$th row in the matrix $\accentset{\circ}{\mathbf{X}}_{\mathscr{S}\cdot \Y}$, we conclude that the matrix $\accentset{\circ}{\mathbf{X}}_{\hat{\mathscr{S}}\cdot \Y}$ is given as
\begin{equation*}
    \accentset{\circ}{\mathbf{X}}_{\hat{\mathscr{S}}\cdot \Y} = \accentset{\circ}{\mathbf{X}}_{\mathscr{S}\cdot \Y}\mathbf{L}_\S \accentset{\circ}{U}_{\hat s-1} \accentset{\circ}{U}_{\hat s-1}^T\mathbf{L}_\S^{-1}.
\end{equation*}
\end{proof}

\begin{proposition}\label{prop:rho_error}
        Let $x\in \mathscr{S}\cdot\Y$ and denote $\varrho = x - \Proj x$. Furthermore, let $D$ be the identity or a differential operator, such that $D \in\{\text{id}, \partial_t, \nabla\}$. Then it holds that
        \begin{equation}\label{eq:rho_min_estimation}
        \left\Vert D \varrho\right\Vert_{\S \cdot \Y} \leq C_D 
        \begin{cases}
            \Sigma_{\hat\Y \rightarrow \hat\S} & \text{if }\Proj = \Proj_{\hat \Y \to \hat{\mathscr{S}}},\\
            \Sigma_{\hat\S \rightarrow \hat\Y} & \text{if }\Proj = \Proj_{\hat{\mathscr{S}} \to \hat \Y},
        \end{cases}
    \end{equation}
    where $\Sigma_{\hat\Y \rightarrow \hat\S}$ and $\Sigma_{\hat\S \rightarrow \hat\Y}$ are defined as in \eqref{eq:singular_value_bounds}.
\end{proposition}

\begin{proof}
    We show the statement for $\Proj = \Proj_{\hat \Y \to \hat{\mathscr{S}}}$. The proof for $\Proj = \Proj_{\hat{\mathscr{S}} \to \hat \Y}$ follows with the same arguments. First of all, it holds that
    \begin{equation}\label{eq:estimate_rho}
            \Vert D \varrho \Vert_{ \S \cdot  \Y} \leq \Vert D(x - \Pi_{\mathscr{S} \cdot \hat \Y}x )\Vert_{\S\cdot\Y} + \Vert D(\Pi_{\mathscr{S} \cdot \hat \Y}x - \Proj_{\hat \Y \to \hat{\mathscr{S}}}x)\Vert_{\S\cdot \Y}.
    \end{equation}

    Depending on which (differential) operator we choose for $D$, we define the matrices
    \begin{equation}\label{eq:space_time_evolution_estimation_matrices}
	\mathbf \Lambda_{\Y} = 
		\begin{cases}
			\mathbf{L}_{\Y} & \text{if } D \in \{\text{id}, \partial_t\},\\ 
			\mathbf{J}_{\Y} & \text{if } D = \nabla,
		\end{cases}\qquad
		\mathbf \Lambda_{\S} = 
		\begin{cases}
			\mathbf{L}_{\S} & \text{if } D \in \{\text{id}, \nabla\},\\ 
			\mathbf{J}_{\S} & \text{if } D =\partial_t.
		\end{cases}
	\end{equation}

    \noindent\textbf{Estimation of space-projection error}\\
    Let $\mathbf{X}$ denote the coefficient matrix of $x\in \mathscr{S}\cdot \Y$. It is known from the proof of \cite[Lemma~2.5]{baumann18} that the coefficient matrix $\mathbf{X}_{\mathscr{S}\cdot \hat\Y}$ of the projection $\Pi_{\mathscr{S}\cdot \hat{\Y}}x$ is given as 
    \begin{equation}\label{eq:similar_representation}
    \mathbf{X}_{\mathscr{S}\cdot \hat\Y} = \mathbf{L}_\Y^{-T} V_{\hat q}V_{\hat q}^T\mathbf{L}_\Y^T \mathbf{X},
    \end{equation}
    where $V_{\hat q}\in\R^{q\times \hat q}$ is the matrix of the $\hat q$ leading left singular vectors of $\mathbf{L}_\Y^T\mathbf{X}\mathbf{L}_\S$. By means of Lemma~\ref{lem:rewrite_x}, we can rewrite $\Vert\cdot\Vert_{\S\cdot\Y}$ with respect to the Frobenius norm and the matrices \eqref{eq:space_time_evolution_estimation_matrices}. It holds that
    \begin{equation*}
    		\begin{split}
			\Vert D (x - \Pi_{\mathscr S \cdot \hat \Y} x) \Vert_{\S\cdot \Y}^2
			&= \Vert \mathbf \Lambda_{\Y}^T (\mathbf X - \mathbf X_{\mathscr S \cdot \hat \Y} ) \mathbf \Lambda_{\S}\Vert_F^2\\
			&=  \Vert \mathbf \Lambda_{\Y}^T \mathbf L_{\Y}^{-T} \mathbf L_{\Y}^{T} (\mathbf X - \mathbf L^{-T}_{\Y} V_{\hat q}V_{\hat q}^T \mathbf L^{T}_{\Y} \mathbf X)\mathbf L_{\S} \mathbf L_{\S}^{-1}\mathbf \Lambda_{\S}\Vert_F^2\\
			&\leq \Vert \mathbf \Lambda_{\Y}^T \mathbf L_{\Y}^{-T}\Vert_F^2 \cdot \Vert \mathbf L_{\Y}^{T} \mathbf X \mathbf L_{\S} - V_{\hat q}V_{\hat q}^T \mathbf L^{T}_{\Y} \mathbf X\mathbf L_{\S} \Vert_F^2 \cdot \Vert \mathbf L_{\S}^{-1}\mathbf \Lambda_{\S}\Vert_F^2\\
		\end{split}
	\end{equation*}
	We denote the singular value decomposition of $\mathbf L_{\Y}^{T} \mathbf X \mathbf L_{\S}$ as $\mathbf L_{\Y}^{T} \mathbf X \mathbf L_{\S} = V \Sigma U^T$ as well as $C_{D} := \Vert \mathbf \Lambda_{\Y}^T \mathbf L_{\Y}^{-T}\Vert_F \cdot \Vert \mathbf L_{\S}^{-1}\mathbf \Lambda_{\S}\Vert_F < \infty$ . Then we have, that
	\begin{equation}\label{eq:estimate_rho_first}
		\begin{split}
			\Vert D (x - \Pi_{\mathscr S \cdot \hat \Y} x) \Vert_{\S\cdot \Y}^2 &\leq C_D^2 \cdot \Vert V \Sigma U^T - V_{\hat q}V_{\hat q}^T V \Sigma U^T\Vert_F^2\\
			&= C_D^2 \cdot \Vert \Sigma -  V^T V_{\hat q}V_{\hat q}^T V \Sigma \Vert_F^2\\
			&= C_D^2 \cdot \Vert \Sigma -  \tilde I_{\hat r} \Sigma \Vert_F^2\\
			&= C_D^2 \cdot \sum_{i = \hat q + 1}^{q} \sigma_i^2,
		\end{split}
	\end{equation}
    where $\{\sigma_i\}_{i=1}^q$ denote the singular values of $\mathbf{L}_\Y^T\mathbf{X}\mathbf{L}_\S$ and we have written
    \begin{equation}\label{eq:strange_identity}
        \tilde I_{\hat q} = \begin{pmatrix}
        I_{\hat q} &\mathbf{0}\\
         \mathbf{0}   & \mathbf{0}
    \end{pmatrix}\in \R^{q\times q}
    \end{equation}
    for $I_{\hat q}\in \R^{\hat q \times \hat q}$ being the $\hat q$-dimensional identity.\\

    \noindent\textbf{Estimation of time-projection error}\\
    The coefficient matrices of the projections in $\Vert D( \Pi_{\mathscr{S}\cdot \hat\Y} x - \Proj_{\hat\Y\to\hat{\mathscr{S}}}x)\Vert_{\S\cdot\Y}$ are denoted as $\mathbf{X}_{\mathscr{S}\cdot \hat\Y}$ and $\mathbf{X}_{\hat{\mathscr{S}}\cdot \hat\Y}$. Due to the treatment of the initial condition as described in Section~\ref{sec:space_time_pod_initial} and the required alteration of the time basis, we have to proceed differently with the estimation of the time-projection error than before for the space-projection. As argued in Remark~\ref{rem:not_orthogonal}, the orthogonality of the reduced basis functions $\{\hat \psi_j\}_{j=1}^{\hat s}$ is not guaranteed and we cannot use a formula similar to \eqref{eq:similar_representation} from the proof of Lemma~\ref{lem:time-basis}.
    
    Since both $\Pi_{\mathscr{S}\cdot \hat\Y} x$ and $\Proj_{\hat\Y\to\hat{\mathscr{S}}} x$ are governed by the same initial condition, the first column of $\mathbf{X}_{\mathscr{S}\cdot \hat\Y}$ is equal to the first column of $\mathbf{X}_{\hat{\mathscr{S}}\cdot \hat\Y}$. Hence, it holds that $\mathbf{X}_{\mathscr{S}\cdot \hat\Y} - \mathbf{X}_{\hat{\mathscr{S}}\cdot \hat\Y} = \accentset{\circ}{\mathbf{X}}_{\mathscr{S}\cdot \hat\Y} - \accentset{\circ}{\mathbf{X}}_{\hat{\mathscr{S}}\cdot \hat\Y},$ where $\accentset{\circ}{\mathbf{X}}_{\mathscr{S}\cdot \hat\Y}$ and $\accentset{\circ}{\mathbf{X}}_{\hat{\mathscr{S}}\cdot \hat\Y}$ result from setting the first column of $\mathbf{X}_{\mathscr{S}\cdot \hat\Y}$ and $\mathbf{X}_{\hat{\mathscr{S}}\cdot \hat\Y}$ to zero. This implies that, after writing $\Vert D(\Pi_{\mathscr{S}\cdot \hat\Y} x - \Proj_{\hat\Y\to\hat{\mathscr{S}}}x)\Vert_{\S\cdot\Y}$ with respect to the Frobenius norm, we can replace $\mathbf{X}_{\mathscr{S}\cdot \hat\Y} - \mathbf{X}_{\hat{\mathscr{S}}\cdot \hat\Y}$ with $\accentset{\circ}{\mathbf{X}}_{\mathscr{S}\cdot \hat\Y} - \accentset{\circ}{\mathbf{X}}_{\hat{\mathscr{S}}\cdot \hat\Y}$. Lemma~\ref{lem:non-orthogonal-coeff-mat-connection} provides an explicit expression for the time-reduced measurement matrix, i.e.\ $\accentset{\circ}{\mathbf{X}}_{\hat{\mathscr{S}}\cdot \hat\Y} = \accentset{\circ}{\mathbf{X}}_{\mathscr{S}\cdot \hat\Y}\mathbf{L}_\S \accentset{\circ}{U}_{\hat s-1} \accentset{\circ}{U}_{\hat s-1}^T\mathbf{L}_\S^{-1}$, where $\accentset{\circ}{U}_{\hat s-1} $ is the matrix of the $\hat s$ leading right singular vectors of $\mathbf{L}_\Y^T \accentset{\circ}{\mathbf{X}}_{\mathscr{S}\cdot \hat \Y}\mathbf{L}_\S$. It follows with Lemma~\ref{lem:rewrite_x} that
    
    \begin{equation*}
        \begin{split}
            \Vert D( \Pi_{\mathscr{S}\cdot \hat\Y} x &- \Proj_{\hat \Y \to \hat{\mathscr{S}}}x)\Vert_{\S\cdot\Y}^2 \\
            &= \Vert \mathbf{\Lambda}_\Y^T [\accentset{\circ}{\mathbf{X}}_{\mathscr{S}\cdot \hat \Y} - \accentset{\circ}{\mathbf{X}}_{\hat{\mathscr{S}}\cdot \hat \Y}]\mathbf{\Lambda}_\S\Vert_F^2\\
            &=  \Vert \mathbf \Lambda_{\Y}^T \mathbf L_{\Y}^{-T} \mathbf L_{\Y}^{T} (\mathbf X_{\mathscr S \cdot \hat \Y} - \mathbf X_{\mathscr S \cdot \hat \Y}\mathbf{L}_{\S}  \accentset{\circ}{U}_{\hat s-1} \accentset{\circ}{U}_{\hat s-1}^T\mathbf{L}_{\S}^{-1})\mathbf \Lambda_{\S}\Vert_F^2\\
	&\leq \Vert \mathbf \Lambda_{\Y}^T \mathbf L_{\Y}^{-T}\Vert_F^2 \cdot \Vert \mathbf L_{\Y}^{T} \mathbf X_{\mathscr S \cdot \hat \Y}  \mathbf L_{\S} -  \mathbf L^{T}_{\Y} \mathbf X_{\mathscr S \cdot \hat \Y}\mathbf{L}_{\S}  \accentset{\circ}{U}_{\hat s-1} \accentset{\circ}{U}_{\hat s-1}^T\Vert_F^2 \cdot \Vert \mathbf L_{\S}^{-1}\mathbf \Lambda_{\S}\Vert_F^2.
        \end{split}
    \end{equation*}
    We insert the singular value decomposition of $\mathbf L_{\Y}^{T} \mathbf X_{\mathscr S \cdot \hat \Y} \mathbf L_{\S}$ as $\mathbf L_{\Y}^{T} \mathbf X_{\mathscr S \cdot \hat \Y} \mathbf L_{\S} = \hat V \hat \Sigma \hat U^T$ and the constant $C_{D}$ from above in this inequality in order to get
	\begin{equation}\label{eq:estimate_rho_second}
		\begin{split}
			\Vert D (\Pi_{\mathscr S \cdot \hat \Y} x - \Pi_{\hat{\mathscr S} \cdot \hat \Y} x) \Vert_{\S\cdot\Y}^2 &\leq C_D^2 \cdot \Vert \hat V \hat \Sigma \hat U^T  -  \hat V \hat \Sigma \hat U^T \accentset{\circ}{U}_{\hat s-1} \accentset{\circ}{U}_{\hat s-1}^T\Vert_F^2\\
			&= C_D^2 \cdot \Vert \hat \Sigma  - \hat \Sigma \hat U^T \accentset{\circ}{U}_{\hat s-1} \accentset{\circ}{U}_{\hat s-1}^T \hat U\Vert_F^2\\
			&= C_D^2 \cdot \Vert \hat \Sigma  - \hat \Sigma \tilde I_{\hat s - 1}\Vert_F^2\\
			&= C_D^2 \cdot \sum_{j = \hat s }^s \accentset{\circ}{\sigma}_j^2,
		\end{split}
	\end{equation}
    where $\tilde I_{\hat s - 1}$ is defined as in \eqref{eq:strange_identity} and $\{\accentset{\circ}{\sigma}_j\}_{i=1}^n$ denote the singular values of $\mathbf{L}_{ \Y}^T\accentset{\circ}{\mathbf{X}}_{\mathscr{S}\cdot \hat \Y}\mathbf{L}_\S$. Inserting \eqref{eq:estimate_rho_first} and \eqref{eq:estimate_rho_second} in \eqref{eq:estimate_rho}, we arrive at
    \begin{equation}\label{eq:estimate_rho_standard}
        \Vert D \varrho\Vert_{\S \cdot \Y} \leq C_D\left( \left( \sum_{i=\hat q + 1}^q \sigma_i^2 \right)^{\frac{1}{2}}+\left( \sum_{j=\hat s}^s \accentset{\circ}{\sigma}_j^2 \right)^{\frac{1}{2}}\right) = C_D\Sigma_{\hat\Y \rightarrow \hat\S}
    \end{equation}
    with $\{\sigma_i\}_{i=1}^s$ being the singular values of $\mathbf{L}_\Y^T\mathbf{X}\mathbf{L}_\S$ and $\{\accentset{\circ}{\sigma}_j\}_{j=1}^s$ the singular values of $\mathbf{L}_{ \Y}^T\accentset{\circ}{\mathbf{X}}_{\mathscr{S}\cdot \hat \Y}\mathbf{L}_\S$.\\
\end{proof}

\subsection{Error bound of $\vartheta$-error}\label{sec:theta-error}
In this section, we derive a bound for the term $\Vert \vartheta\Vert_{\S\cdot \Y}$ in Proposition~\ref{prop:bound_evolution_problem}. We remark, that parts of the proof are motivated by the computations in the proof of \cite[Proposition~2.36]{steih14}.\\

\begin{proposition}\label{prop:bound_evolution_problem}
    Let $\varrho$ and $\vartheta$ be defined as in \eqref{eq:theta_rho}. Then it holds that
    \begin{equation}
        \Vert \vartheta\Vert_{\S\cdot\Y} \leq C_\vartheta \cdot \left(\Vert \varrho\Vert_{\S \cdot \Y} + \Vert \varrho_t\Vert_{\S \cdot \Y} + \Vert \nabla \varrho\Vert_{\S \cdot \Y}\right)
        \end{equation}
    for a constant $C_\vartheta>0$.
\end{proposition}
\begin{proof}
    Consider the bilinear form $B(\cdot,\cdot)$ and linear form $L(\cdot)$ as introduced in \eqref{eq:bilinear_linear_form}. We use the coercivity of $a(t;\cdot, \cdot)$ as stated in \eqref{eq:a_coercive_continuous:b}, as well as the fact that $\vartheta(0) = 0$ from \eqref{eq:teta_zero_zero} in order to show that
    \begin{equation}\label{eq:proof_theta_1}
        \begin{split}
            B(\vartheta, \vartheta) &= \int_0^T \langle(\vartheta(t))_t, \vartheta(t)\rangle_{V^*,V} \intd t + \int_0^T a(t; \vartheta, \vartheta)\intd t\\
            &= \frac{1}{2} \left[\Vert \vartheta(T)\Vert_H^2 - \Vert \vartheta(0)\Vert^2_H\right] + \int_0^T a(t;\vartheta,\vartheta)\intd t\\
            &\geq \int_0^T a(t; \vartheta, \vartheta)\intd t\\
            &\geq \alpha \int_0^T\Vert \vartheta(t) \Vert_V^2 \intd t\\
            &= \alpha \Vert \vartheta \Vert_{L^2(0,T;V)}^2
        \end{split}
    \end{equation}
    for $\vartheta \in \hat{\mathscr{S}} \cdot \hat \Y$. Next, the Galerkin-orthogonality with respect to $B(\cdot, \cdot)$ holds as
    \begin{equation}\label{eq:proof_theta_2}
        \begin{split}
            B(\vartheta, \vartheta) &= B(\Proj x, \vartheta) - B(\hat x, \vartheta)\\
            &= B(\Proj x, \vartheta) - B(x, \vartheta) + B(x, \vartheta) - B(\hat x, \vartheta)\\
            &= B(\Proj x - x, \vartheta) + L(\vartheta) - L(\vartheta)\\
            &= - B(\varrho, \vartheta).
        \end{split}
    \end{equation}
    Taking advantage of the Cauchy-Schwarz inequality as well as the uniform continuity of $a(t;\cdot,\cdot)$ as stated in \eqref{eq:a_coercive_continuous:a}, we estimate
    \begin{equation}\label{eq:proof_theta_3}
        \begin{split}
            |B(\varrho, \vartheta)| &\leq \int_0^T \Vert\varrho_t(t)\Vert_{V^*} \Vert \vartheta (t)\Vert_V \intd t + \gamma \int_0^T\Vert \varrho(t) \Vert_V \Vert\vartheta(t)\Vert_V\intd t\\
            &\leq \max\{1,\gamma\} \int_0^T \left(\Vert\varrho_t(t)\Vert_{V^*} + \Vert \varrho(t) \Vert_V\right) \cdot \Vert\vartheta(t)\Vert_V\intd t\\
            &\overset{\text{Hölder}}{\leq} \max\{1,\gamma\} \left(\int_0^T\left(\Vert\varrho_t(t)\Vert_{V^*} + \Vert \varrho(t) \Vert_V\right)^2\intd t\right)^{\frac{1}{2}} \left(\int_0^T \Vert \vartheta(t)\Vert_V^2\intd t\right)^{\frac{1}{2}}\\
            &\leq \max\{1,\gamma\} \left(\int_0^T 2 \left(\Vert\varrho_t(t)\Vert_{V^*}^2 + \Vert \varrho(t) \Vert_V^2\right)\intd t\right)^{\frac{1}{2}} \Vert \vartheta \Vert_{L^2(0,T;V)}\\
            &\leq \sqrt{2} \max \{1,\gamma\} \left(\Vert \varrho_t\Vert_{L^2(0,T;V^*)} + \Vert \varrho\Vert_{L^2(0,T;V)}\right)\Vert \vartheta \Vert_{L^2(0,T;V)}
        \end{split}
    \end{equation}
    Combining \eqref{eq:proof_theta_1}, \eqref{eq:proof_theta_2} and \eqref{eq:proof_theta_3}, we find that
    \begin{equation}
        \begin{split}
        \alpha \Vert \vartheta \Vert^2_{L^2(0,T;V)}&\overset{\eqref{eq:proof_theta_1}}{\leq} B(\vartheta, \vartheta)\overset{\eqref{eq:proof_theta_2}}{=} - B(\varrho, \vartheta)\\
        &\overset{\eqref{eq:proof_theta_3}}{\leq} \sqrt 2 \max\{1,\gamma\}\left(\Vert \varrho_t\Vert_{L^2(0,T;V^*)} + \Vert \varrho\Vert_{L^2(0,T;V)}\right)\Vert \vartheta \Vert_{L^2(0,T;V)}.
        \end{split}
    \end{equation}
    We divide both hand sides by $\alpha \Vert \vartheta \Vert_{L^2(0,T;V)}$ and insert $V=H^1_0(\Omega)$. Furthermore, the fact that $\varrho_t \in L^2(0,T;L^2(\Omega))$ (instead of just $\varrho_t \in L^2(0,T;H^{-1}(\Omega))$) allows to estimate the dual norm, such that
    \begin{equation*}
    	\Vert \vartheta \Vert_{L^2(0,T;H^1_0(\Omega))} \leq \tilde C_{\vartheta}\left(\Vert \varrho_t\Vert_{\S \cdot \Y} + \Vert \varrho\Vert_{\S \cdot \Y} + \Vert \nabla \varrho\Vert_{\S \cdot \Y}\right)
    \end{equation*}
    for $\tilde C_\vartheta > 0$. An application of Poincaré's inequality yields the claim.\\
\end{proof}

\subsection{Combination of error approximation results}\label{sec:error-summary}

We combine the results from the previous subsections in order to prove the following error estimate.\\

\begin{theorem}[Error estimation]\label{thm:final_estimation}
    Let $x\in \mathscr{S}\cdot \Y$ be a solution to \eqref{eq:fom} and $\hat x \in \hat{\mathscr{S}}\cdot \hat\Y$ a solution to \eqref{eq:rom}. Then it holds that
    \begin{equation}
        \Vert x - \hat x\Vert_{\S\cdot \Y} \leq C \cdot \begin{cases}
            \Sigma_{\hat\Y \rightarrow \hat{\S}} & \text{if }\Proj = \Proj_{\hat \Y \to \hat{\mathscr{S}}},\\
             \Sigma_{\hat\S \rightarrow \hat\Y} & \text{if }\Proj = \Proj_{\hat{\mathscr{S}} \to \hat \Y},
        \end{cases}
    \end{equation}
    where 
    \begin{equation*}
     \Sigma_{\hat\Y \rightarrow \hat\S} = \left( \sum_{i=\hat q + 1}^q \sigma_i^2\right)^{\frac{1}{2}} + \left(\sum_{j=\hat s}^s \accentset{\circ}{\sigma}_j^2 \right)^{\frac{1}{2}},\quad
    \Sigma_{\hat\S \rightarrow \hat\Y} = \left(\sum_{j=\hat s}^s \underline{\accentset{\circ}{\sigma}}_j^2\right)^{\frac{1}{2}} + \left(\sum_{i=\hat q + 1}^q \smash{\underline \sigma}_i^2\right)^{\frac{1}{2}}.
    \end{equation*}
    In this, $\{\sigma_j\}_{j=1}^s$ denote the singular values of $\mathbf{L}_\Y^T\mathbf{X}\mathbf{L}_\S$, $\{\accentset{\circ}{\sigma}_j\}_{j=1}^s$ the singular values of $\mathbf{L}_{ \Y}^T\accentset{\circ}{\mathbf{X}}_{\mathscr{S}\cdot \hat \Y}\mathbf{L}_\S$, $\{\underline{\accentset{\circ}{\sigma}}_j\}_{j=1}^s$ the singular values of $\mathbf{L}_{ \Y}^T\accentset{\circ}{\mathbf{X}}\mathbf{L}_{\S}$ and $\{\smash{\underline \sigma}_i\}_{i=1}^q$ the singular values of $\mathbf{L}_{ \Y}^T\mathbf{X}_{\hat{\mathscr{S}} \cdot \Y}\mathbf{L}_{\S}$. The matrices $\mathbf{X}$, $\mathbf{X}_{\mathscr{S} \cdot \hat{\Y}}$ and $\mathbf{X}_{\hat{\mathscr{S}} \cdot \Y}$ denote the coefficient matrices of $x$, $\Pi_{\mathscr{S} \cdot \hat{\Y}}x$ and $\Pi_{\hat{\mathscr{S}} \cdot \Y}x$. An overset "$\circ$" indicates the alteration of those matrices by setting the first column to zero.
\end{theorem}
\begin{proof}
    We show the statement for $\Proj = \Proj_{\hat \Y \to \hat{\mathscr{S}}}$. The proof for $\Proj = \Proj_{\hat{\mathscr{S}} \to \hat \Y}$ follows with the same arguments. Using $\vartheta$ and $\varrho$ as defined in \eqref{eq:theta_rho}, we estimate
    $$\Vert x - \hat x\Vert_{\S\cdot \Y} = \Vert \varrho + \vartheta \Vert_{\S\cdot\Y} \leq \Vert \varrho \Vert_{ \S \cdot  \Y} + \Vert \vartheta \Vert_{\S\cdot\Y}.$$
    The application of Proposition~\ref{prop:bound_evolution_problem} yields that
    $$\Vert x - \hat x\Vert_{\S\cdot \Y} \leq \Vert \varrho\Vert_{\S\cdot\Y} + C_\vartheta \left(\Vert \varrho\Vert_{\S\cdot\Y} + \Vert \varrho_t\Vert_{\S \cdot \Y} + \Vert \nabla \varrho\Vert_{\S \cdot \Y}  \right),$$
    with $C_\vartheta > 0$. Next, we apply Proposition~\ref{prop:rho_error} in order to estimate the norms on the right hand sides. This results in
    $$\Vert x - \hat x\Vert_{\S\cdot \Y} \leq  \left(C_{\text{id}} + C_\vartheta ( C_{\text{id}}+C_{\partial_t} + C_\nabla )\right) \Sigma_{\hat\Y \rightarrow \hat\S}$$
    and the claim follows with $C := C_{\text{id}} + C_\vartheta ( C_{\text{id}}+C_{\partial_t} + C_\nabla )$.\\
\end{proof}

\begin{remark}\label{rem:sigma_not_zero}
The sums in $\Sigma_{\hat\Y \rightarrow \hat\S}$ and $\Sigma_{\hat\S \rightarrow \hat\Y}$ which sum up the singular-values coming from the reduction in time, start with the index $j = \hat s$ instead of $j=\hat s +1$ as one may expect it from the reduction in space for which the sums start at $i = \hat q + 1$. Hence, if we choose $\hat s = s$ and $\hat q = q$, the error component related to the reduction in space vanishes, while the error component related to the reduction in time does not vanish. This is consistent with the intuition, that the error cannot vanish, since the basis is not optimal anymore due to the alteration of the optimal basis in Section~\ref{sec:space_time_pod_initial} in order to include the initial condition. The remaining error term is equal to the squared singular value which is associated with the singular vector which is cut off in \eqref{eq:different_U}.
\end{remark}

\section{Numerical examples}\label{sec:numerical examples}
In this section, we investigate the error bound numerically for two exemplary problems. While doing so, we will primarily focus on the intuition behind the error bounds and refer for instructions on how to apply space-time POD to \cite[Sections 3 \& 4]{baumann18}, where the numerical implementation of the given problems has been discussed in detail. In Example~\ref{ex:simple} we will focus on the investigation of the error in space-time POD for a specific choice of the reduced dimensions $\hat s$ and $\hat q$. Our goal is to demonstrate the similarity of space-time POD and standard POD with respect to the decreasing behavior of the singular values and its effect to the reduced bases. In Example~\ref{ex:circle}, we will investigate the behavior of the error bound and related terms for changing reduced dimensions $\hat s$ and $\hat q$. In both examples, we apply space-time POD for the problem
\begin{equation}\label{eq:numerical_problem}
    \left\{\begin{array}{rll}
        x_t - \mu\Delta x &= f & \text{in }\Omega_T,\\
        x &= 0 & \text{on }\Sigma_T,\\
        x(0) &= x_0 & \text{in }\Omega,
    \end{array}\right.
\end{equation}
where $[0,T] = [0,1]$, $\Omega = [0,1]\subset \R^1$, $\Omega_T = [0,T]\times \Omega $, $\Sigma_T = [0,T] \times \partial \Omega$ and $\mu > 0$. Problem \eqref{eq:numerical_problem} is a special case of \eqref{eq:evolution_problem} for the choice of $\A(t) = - \mu \Delta$ and we use the space-time variational formulation \eqref{eq:space_time_variational} for the numerical implementation of it. The bilinear and linear form in the variational formulation are in this case given as
\begin{equation*}
    \begin{split}
        B(x,\varphi) &= \int_{\Omega_T} x_t \varphi + \mu \nabla x \cdot \nabla \varphi \intd x \intd t,\\
        L(\varphi) &= \int_{\Omega_T}f \varphi \intd x \intd t.
    \end{split}
\end{equation*}
The computations will be executed on an equidistant grid in space and time with $\Delta \tau = \Delta \xi = 0.01$. For both the computation of the FOM and the ROM, we use a tensor product approach, which has also been used for the ROM in \cite[Sections 3 \& 4]{baumann18}. While \cite{baumann18} used a method of lines approach for the computation of the FOM, we use the tensor product approach for the computation of the FOM as well. The tensor product approach as well as the derivation of the system matrices of the ROM is explained in \cite[Section~3]{baumann18}. The implementations of both examples have been realized in Python. The code is available at \url{https://doi.org/10.5281/zenodo.19710435}.\\

\begin{example}\label{ex:simple}
    We start with an example in which we investigate the error behavior of space-time POD for a fixed choice of $\hat s$ and $\hat q$. In the setup of this example, we choose $\mu = 0.4$. Furthermore, we choose the initial condition $x_0$ and the right hand side $f(\tau, \xi)$, such that a solution to \eqref{eq:numerical_problem} is given by
    $$x(\tau, \xi) = 10 \cdot \sin(\pi\xi - 2\tau) \cdot \cos( \pi \tau-3\xi)\cdot \xi \cdot (\xi-1)^3.$$
    The computed FOM is depicted in Figure~\ref{fig:ex_simple_dynamics}~(a). Furthermore, the singular values of $\mathbf{L}_\Y^T\mathbf{X}\mathbf{L}_\S$, $\mathbf{L}_{ \Y}^T\accentset{\circ}{\mathbf{X}}_{\mathscr{S}\cdot \hat \Y}\mathbf{L}_\S$, $\mathbf{L}_{ \Y}^T\accentset{\circ}{\mathbf{X}}\mathbf{L}_{\S}$ and $\mathbf{L}_{ \Y}^T\mathbf{X}_{\hat{\mathscr{S}} \cdot \Y}\mathbf{L}_{\S}$ are shown in Figure~\ref{fig:ex_simple_results}~(a). Notice, that the singular values of $\mathbf{L}_{ \Y}^T\accentset{\circ}{\mathbf{X}}_{\mathscr{S}\cdot \hat \Y}\mathbf{L}_\S$ and $\mathbf{L}_{ \Y}^T\mathbf{X}_{\hat{\mathscr{S}} \cdot \Y}\mathbf{L}_{\S}$ are dependent on the reduced dimensions $\hat q$ and $\hat s$ and therefore depicted for the choice of $\hat q = 5$ and $\hat s = 15$. While the $\ell$-th singular value of $\mathbf{L}_{ \Y}^T\accentset{\circ}{\mathbf{X}}_{\mathscr{S}\cdot \hat \Y}\mathbf{L}_\S$ vanishes for $\ell \geq \hat q$, the $\ell$-th singular value of $\mathbf{L}_{ \Y}^T\mathbf{X}_{\hat{\mathscr{S}} \cdot \Y}\mathbf{L}_{\S}$ already vanishes for $\ell \geq \hat s -1$ instead of $\ell \geq \hat s$ as one might expect it. The reason for this behavior is the different construction of the time-reduced basis due to the incorporation of the initial condition and results from the proof of Proposition~\ref{prop:rho_error}, where the sum of squared singular values starts with $\hat s$ instead of $\hat s + 1$. This phenomenon can 
    be observed even better in Example~\ref{ex:circle}.
    
    From the decays of the singular values, we conclude that a good choice for the reduced dimension such that only a reasonable amount of information is cut off, is $\hat s =20$ and $\hat q = 20$. Based on this choice, we compute the projection of the solution of the FOM onto the reduced space as well as the corresponding ROM. The error between the solution to the FOM and the projection is depicted in Figure~\ref{fig:ex_simple_dynamics}~(b), while the error between the solution to the FOM and the ROM is depicted in Figure~\ref{fig:ex_simple_dynamics}~(c). The reduced bases in space and time are depicted in Figure~\ref{fig:ex_simple_results}~(b) and (c) respectively. Notice, that for the reduced time basis, the first basis vector in $\hat \Psi$ is chosen as described in Section~\ref{sec:space_time_pod_initial} in order to treat the initial condition. In both $\hat \Psi$ and $\hat \Upsilon$, the basis vectors start to fluctuate at some point. For the reduced time basis $\hat \Psi$, this happens for the sixth basis function. For the reduced space basis, the fluctuations initially enter the basis functions more unnoticed. They cause some fluctuations at the spatial boundaries until they are clearly visible for the 23rd basis function. This behavior does not only occur for space-time POD, but is also known for standard POD, see e.g.\ \cite[Section~5]{gubisch17}. We expect that those fluctuations may cause complications during the computation of the ROM when fluctuating basis functions belonging to small singular values enter the system of equations which is solved when computing the ROM.\\

    \begin{figure}
        \centering
            \begin{subfigure}[t]{0.32\textwidth}
                \centering
                \includegraphics[width=\linewidth]{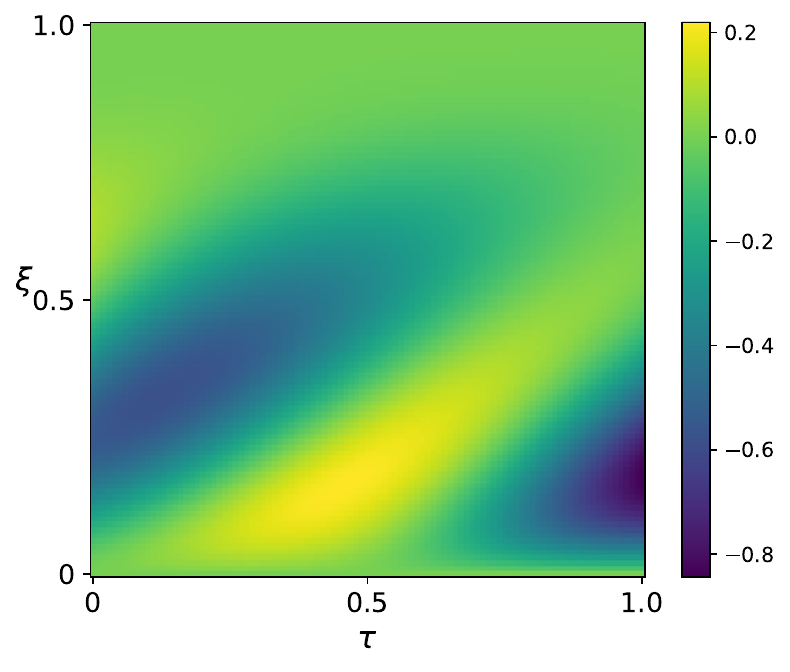}
                \caption{FOM $x(\tau, \xi)$}
            \end{subfigure}
            \hspace*{\fill}
            \begin{subfigure}[t]{0.32\textwidth}
                \centering
                \includegraphics[width=\linewidth]{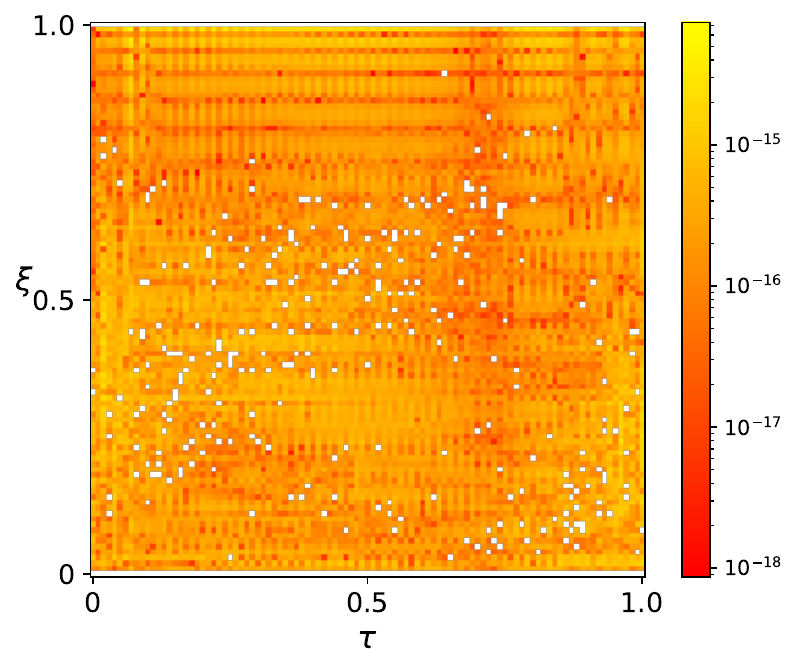}
                \caption{$| x(\tau_j, \xi_i) - (\Proj x)(\tau_j, \xi_i)|$}
            \end{subfigure}
            \hspace*{\fill}
            \begin{subfigure}[t]{0.32\textwidth}
                \centering
                \includegraphics[width=\linewidth]{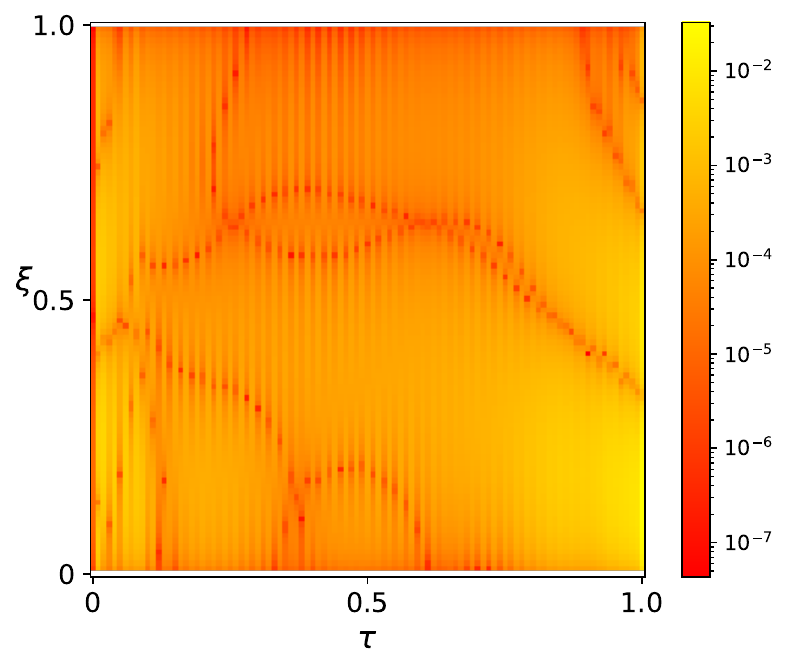}
                \caption{$| x(\tau_j, \xi_i) - \hat x(\tau_j, \xi_i)|$}
            \end{subfigure}

        \caption{Depiction of the full-order dynamics $x$ from Example~\ref{ex:simple} as well as the error between the $x$ and the projection onto $\hat{\mathscr{S}}\cdot \hat \Y$, as well as the error between $x$ and the ROM $\hat x$. The projection has been chosen as $\Proj =\Proj_{\hat \Y \to \hat{\mathscr{S}}}$ and the reduced dimensions as $\hat s = 20$ and $\hat q = 20$. Notice that the plots in (b) and (c) are logarithmically scaled. The white dots in (b) represent points where the error vanishes.}
        \label{fig:ex_simple_dynamics}.
    \end{figure}

    \begin{figure}
        \centering
            \begin{subfigure}[t]{0.32\textwidth}
                \centering
                \includegraphics[width=\linewidth]{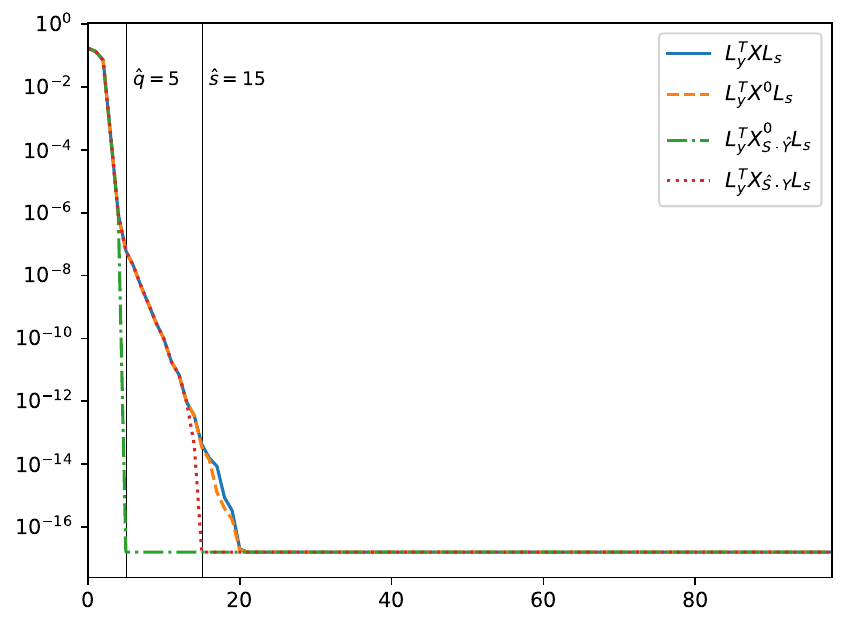}
                \caption{Singular value descent of $\mathbf{L}_\Y^T\mathbf{X}\mathbf{L}_\S$, $\mathbf{L}_{\Y}^T\protect\accentset{\circ}{\mathbf{X}}_{\mathscr{S}\cdot \hat \Y}\mathbf{L}_\S$, $\mathbf{L}_{ \Y}^T\protect\accentset{\circ}{\mathbf{X}}\mathbf{L}_{\S}$ and $\mathbf{L}_{ \Y}^T\mathbf{X}_{\hat{\mathscr{S}} \cdot \Y}\mathbf{L}_{\S}$ for $\hat s = 15$ and $\hat q = 5$.}
            \end{subfigure}
            \hspace*{\fill}
            \begin{subfigure}[t]{0.32\textwidth}
                \centering
                \includegraphics[width=\linewidth]{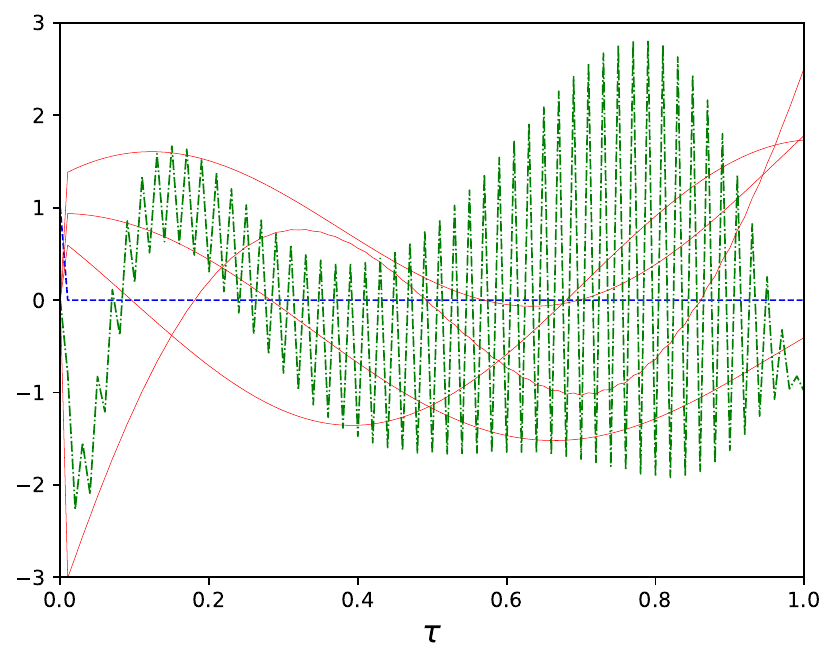}
                \caption{First six basis functions of $\Psi$ with the first basis function being highlighted in dashed blue and the sixth basis function being highlighted in dashdotted green.}
            \end{subfigure}
            \hspace*{\fill}
            \begin{subfigure}[t]{0.32\textwidth}
                \centering
                \includegraphics[width=\linewidth]{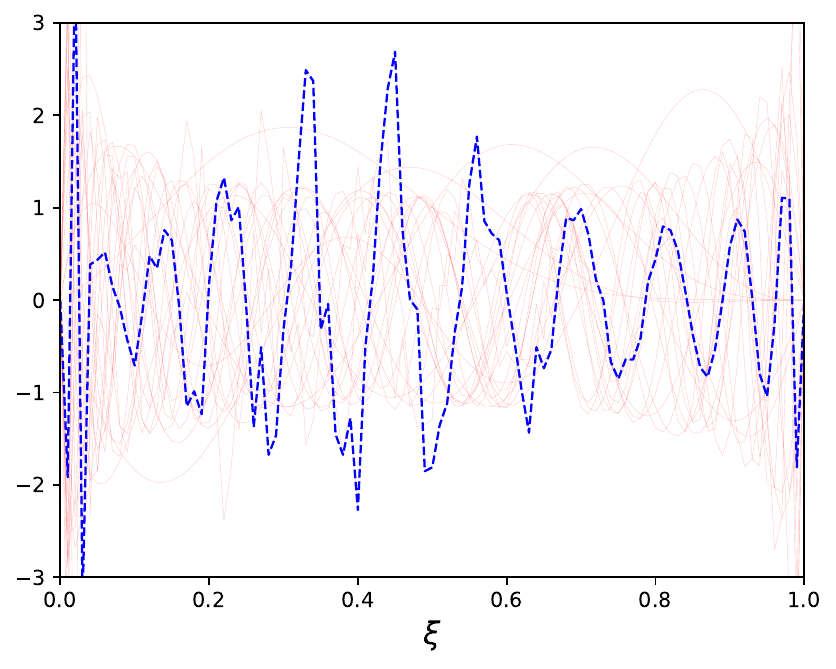}
                \caption{First $23$ basis functions of $\Upsilon$ with the $23$st basis function being highlighted in dashed blue.}
            \end{subfigure}

        \caption{Insights in the reduction process executed in Example~\ref{ex:simple}. Depicted are the decays of the singular values of the matrices, which are involved in the computation of the reduced bases in (a). Furthermore, the first elements of the reduced basis in time $\hat \Psi$ and the reduced basis in space $\hat \Upsilon$ are depicted. They are depicted in (b) and (c) and have been computed separately from each other from $\Psi$ and $\Upsilon$ using Lemma~\ref{lem:space-basis} and Lemma~\ref{lem:time-basis}, where the treatment of the initial condition has been used as described in Section~\ref{sec:space_time_pod_initial} has been applied.}
        \label{fig:ex_simple_results}.
    \end{figure}

\end{example}

\begin{example}\label{ex:circle}
    In this example, we choose $\mu = 1$, the initial condition $x_0 = 0$ and $f(\tau,\xi)= 5 \cdot \chi_D$, where
    $$ \chi_D(\tau, \xi) := \begin{cases}
        1 & \text{if }(\tau, \xi)\in D,\\
        0 & \text{if }(\tau, \xi)\not\in D,
    \end{cases}$$ 
    denotes the indicator function and $D$ a ring in the space-time domain, i.e.\ $D = B_{1/\sqrt 5}(1/2) \setminus B_{1/\sqrt 8}(1/2) \subset [0,T]\times \Omega$.
    The right hand side $f(\tau,\xi)$ and the computed FOM solution $x(\tau,\xi)$ are depicted in Figure~\ref{fig:ex_circle_f_func_f} (a) and (b). As an indicator for the effectiveness of the reduction with space-time POD, we compute the singular values of $\mathbf{L}_\Y^T\mathbf{X}\mathbf{L}_\S$, $\mathbf{L}_{ \Y}^T\accentset{\circ}{\mathbf{X}}_{\mathscr{S}\cdot \hat \Y}\mathbf{L}_\S$, $\mathbf{L}_{ \Y}^T\accentset{\circ}{\mathbf{X}}\mathbf{L}_{\S}$ and $\mathbf{L}_{ \Y}^T\mathbf{X}_{\hat{\mathscr{S}} \cdot \Y}\mathbf{L}_{\S}$, see Figure~\ref{fig:ex_circle_f_func_f}~(c). Notice that the decays of the singular values are, as described in Example~\ref{ex:simple}, dependent on the reduced dimensions $\hat s$ and $\hat q$. In Figure~\ref{fig:ex_circle_f_func_f}~(c), the decays are depicted for the choice of $\hat s = 40$ and $\hat q = 20$. The downward jumping behavior of the curves can be explained with the same arguments as in Example~\ref{ex:simple}.

        \begin{figure}
            \centering
            \begin{subfigure}[t]{0.3\textwidth}
                \centering
                \includegraphics[width=\linewidth]{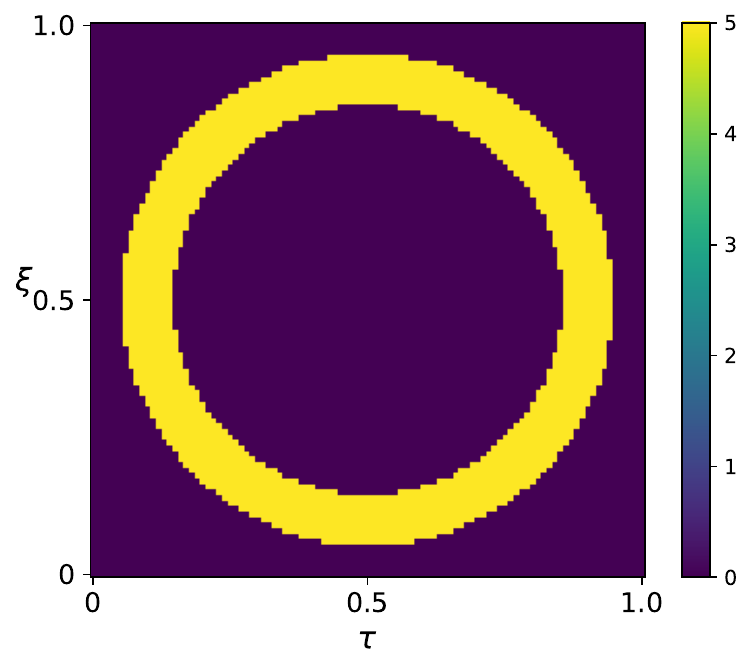}
                \caption{$f(\tau, \xi)$. The value $0$ is depicted in dark blue / violet, while the value $5$ is depicted in light yellow.}
            \end{subfigure}
            \hspace*{\fill}
            \begin{subfigure}[t]{0.31\textwidth}
                \centering
                \includegraphics[width=\linewidth]{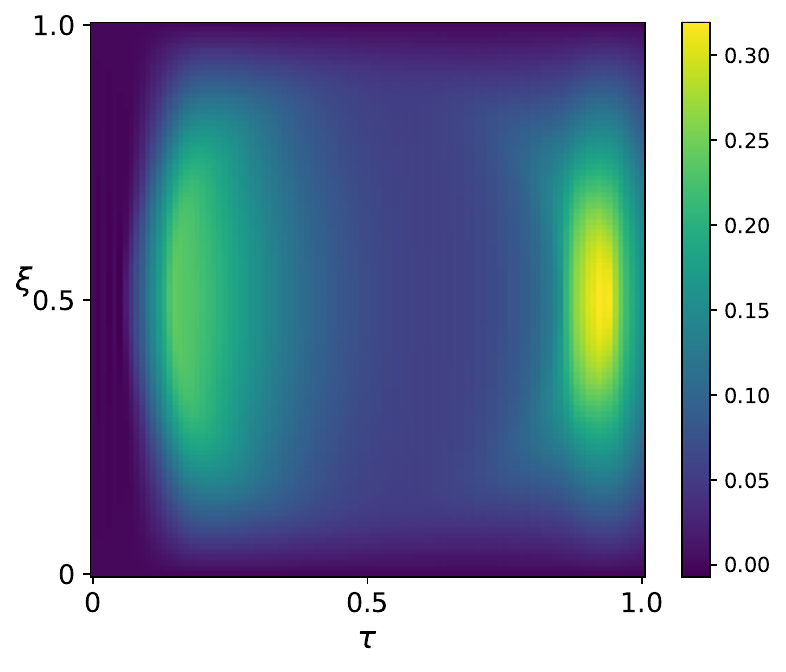}
                \caption{Solution to the FOM computed with finite elements in space-time.}
            \end{subfigure}
            \begin{subfigure}[t]{0.36\textwidth}
                \centering
                \includegraphics[width=\linewidth]{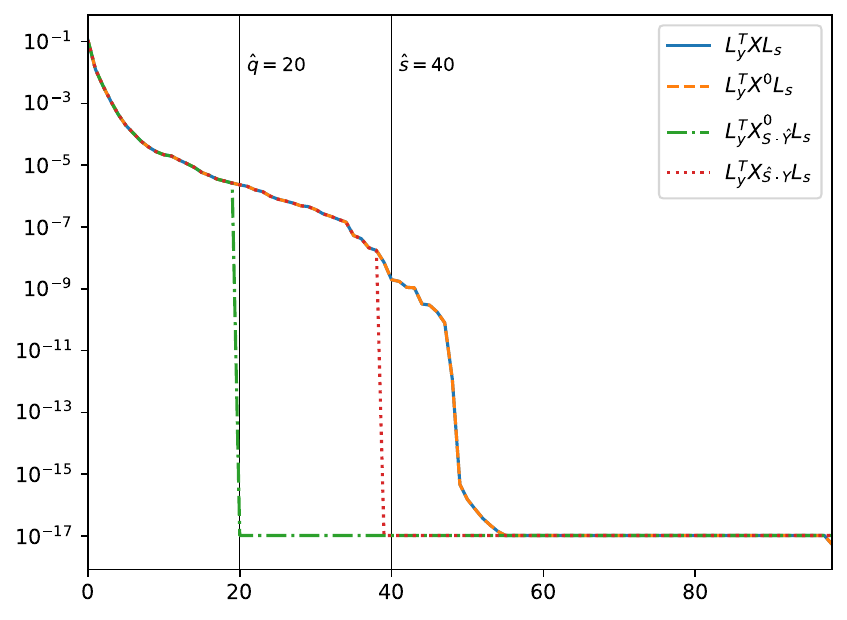}
                \caption{Singular value descent of $\mathbf{L}_\Y^T\mathbf{X}\mathbf{L}_\S$, $\mathbf{L}_{ \Y}^T\accentset{\circ}{\mathbf{X}}_{\mathscr{S}\cdot \hat \Y}\mathbf{L}_\S$, $\mathbf{L}_{ \Y}^T\accentset{\circ}{\mathbf{X}}\mathbf{L}_{\S}$ and $\mathbf{L}_{ \Y}^T\mathbf{X}_{\hat{\mathscr{S}} \cdot \Y}\mathbf{L}_{\S}$ for $\hat q = 20$ and $\hat s = 40$.}
            \end{subfigure}
            \caption{Visualization of the problem setup in Example~\ref{ex:circle}.}
            \label{fig:ex_circle_f_func_f}
        \end{figure}

    The error bound $\Sigma_{\hat\Y \rightarrow \hat\S}$, the error between the solution to the FOM $x$ and its projection onto $\Y \to \hat{\mathscr{S}}$, i.e.\ $\Proj_{\hat \Y \to \hat{\mathscr{S}}} x$, as well as the error between the solution to the FOM and the ROM are displayed in dependence of the reduced dimensions $\hat s$ and $\hat q$ in Figure~\ref{fig:ex_circle_error_comp} and Figure~\ref{fig:ex_circle_error_comp_all}.

    \begin{figure}
        \centering
            \begin{subfigure}[t]{0.32\textwidth}
                \centering
                \includegraphics[width=\linewidth]{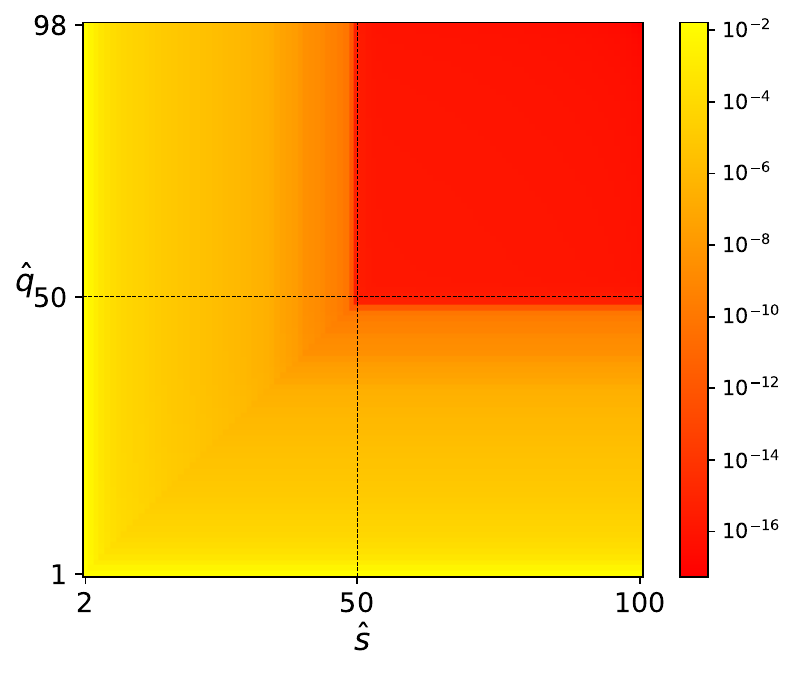}
                \caption{$\Sigma_{\hat\Y \rightarrow \hat\S}$}
            \end{subfigure}
            \hspace*{\fill}
            \begin{subfigure}[t]{0.32\textwidth}
                \centering
                \includegraphics[width=\linewidth]{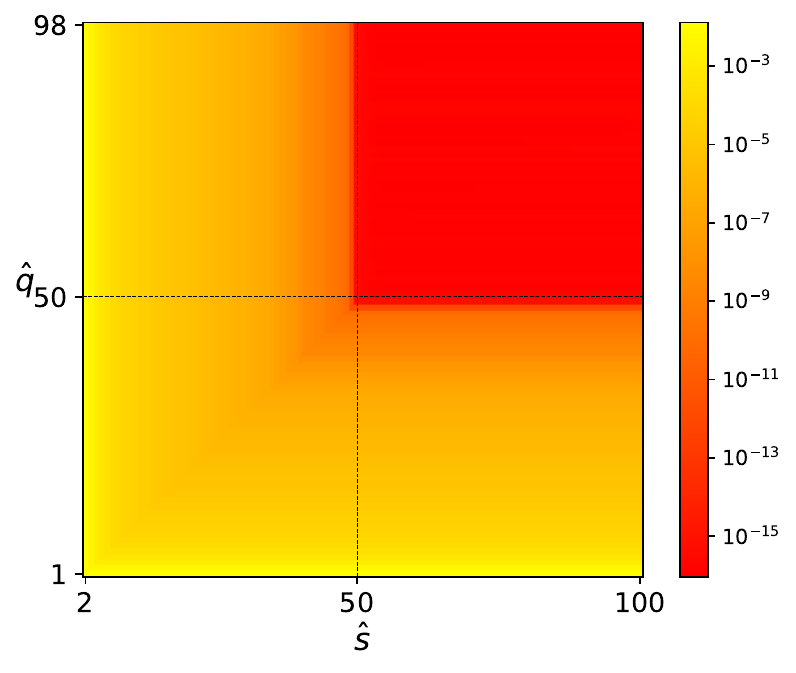}
                \caption{$\Vert x - \Proj x\Vert_{\S\cdot \Y}$}
            \end{subfigure}
            \hspace*{\fill}
            \begin{subfigure}[t]{0.32\textwidth}
                \centering
                \includegraphics[width=\linewidth]{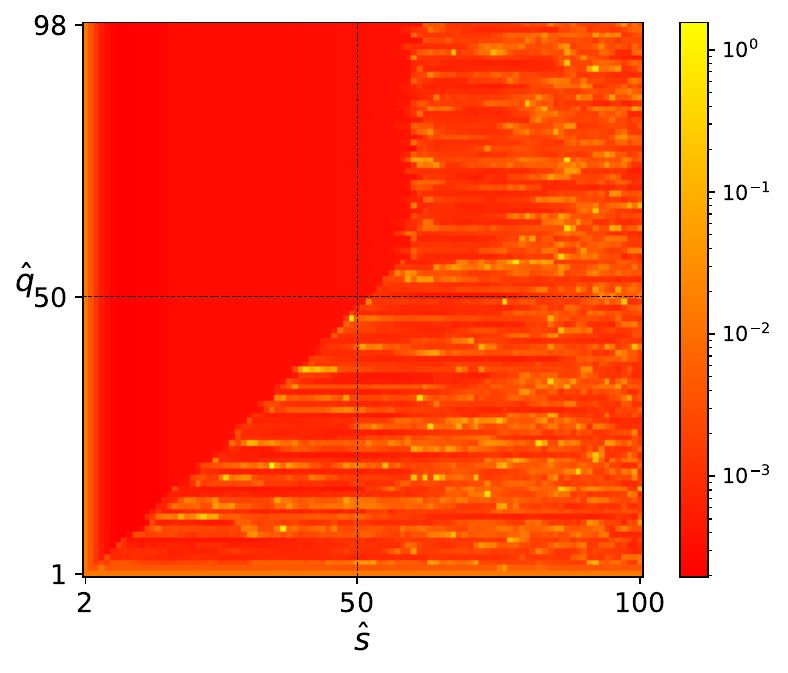}
                \caption{$\Vert x - \hat x\Vert_{\S\cdot \Y}$}
            \end{subfigure}

        \caption{Visualization of the error terms in Example~\ref{ex:circle}. We investigate $\Sigma_{\hat\Y \rightarrow \hat\S}$, $\Vert x - \Proj x\Vert_{\S\cdot \Y}$ and $\Vert x - \hat x\Vert_{\S\cdot \Y}$ in case of $\Proj = \Proj_{\hat \Y \to \hat{\mathscr{S}}}$. In each plot, those error values are color-coded logarithmically in each point for a different choice of the reduced dimensions $\hat s$ and $\hat q$. In each plot, lines indicate the case where $\hat s = \hat q = 50$. This serves as an orientation during a comparison with the values in Figure~\ref{fig:ex_circle_error_comp_all}.}
        \label{fig:ex_circle_error_comp}
    \end{figure}
    
    \begin{SCfigure}
    \includegraphics[width=0.5\linewidth]{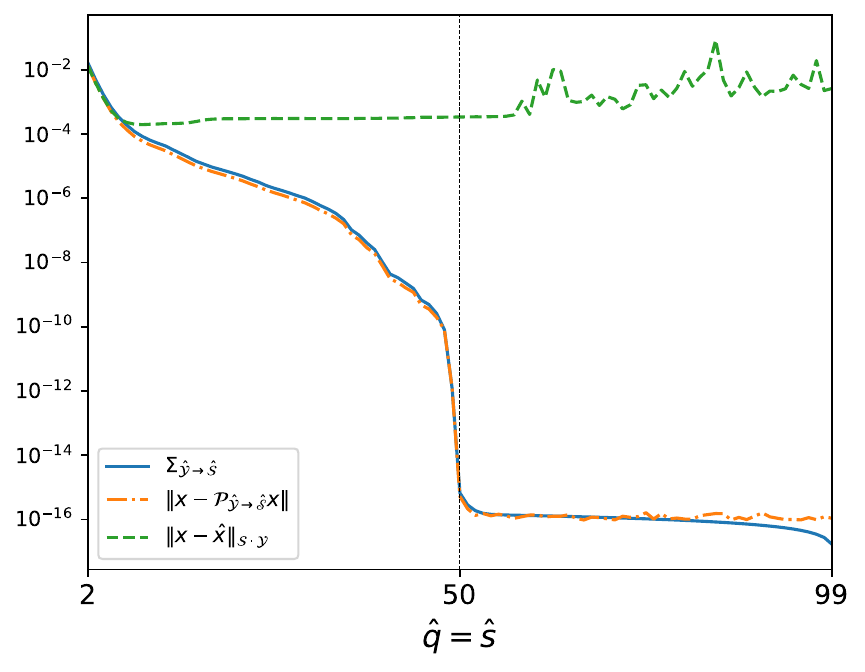}
    \caption{Visualization of the error terms in Example~\ref{ex:circle}. We investigate $\Sigma_{\hat\Y \rightarrow \hat\S}$, $\Vert x - \Proj x\Vert_{\S\cdot \Y}$ and $\Vert x - \hat x\Vert_{\S\cdot \Y}$ in case of $\Proj = \Proj_{\hat \Y \to \hat{\mathscr{S}}}$. Depicted are those error values for the reduced dimensions with $\hat s = \hat q$. The line which indicates the case where $\hat s = \hat q = 52$ serves as an orientation during a comparison with the values in Figure~\ref{fig:ex_circle_error_comp}.\\\,\\}
    \label{fig:ex_circle_error_comp_all}
    \end{SCfigure}

    In the numerical analysis, the projection error is bounded from above by the sum of the singular values $\Sigma_{\hat\Y \rightarrow \hat\S}$. Both of those terms decrease continuously until they reach machine precision. The error occurring during the computation in the ROM first decreases but then stagnates after just a few basis functions. The error increases and shows some fluctuating behavior after the reduced dimensions are chosen to be so large that basis functions belonging to singular values which are close to machine precision enter the ROM system with which the solution to the ROM is computed. The fluctuating behavior which the basis functions belonging to such singular values show, has also already been observed and visualized in Example~\ref{ex:simple}. This phenomenon is not specific for space-time POD, but also occurs during the reduction with standard POD, see e.g.\ \cite{behzad15}.

    This argumentation also explains the triangular shape in the error between the FOM and the ROM for different choices of $\hat s$ and $\hat q$ in Figure~\ref{fig:ex_circle_error_comp}~(c). Notice that the plot shows the errors for $\Proj = \Proj_{\hat \Y \to \hat{\mathscr{S}}}$. Hence, the projection projects the space dimension first and the time dimension afterwards. As described earlier and depicted in Figure~\ref{fig:ex_circle_f_func_f}~(c), the singular values of $\mathbf{L}_{ \Y}^T\accentset{\circ}{\mathbf{X}}_{\mathscr{S}\cdot \hat \Y}\mathbf{L}_\S$ (the matrix on which the followed reduction in time is based upon) drop to machine precision after $\hat q$ singular values. Hence, if the reduced dimension in temporal direction $\hat s$ is chosen larger than the reduced dimension $\hat q$ in spatial direction, singular vectors belonging to singular values with the magnitude of machine precision enter the computation of the ROM and get amplified. Notice that this suggests to choose $\hat s \approx \hat q$.
    
    Finally, we accomplish to get an insight into the magnitude of the constants from Proposition~\ref{prop:rho_error}, on which the constant $C$ in the error estimation in Theorem~\ref{thm:final_estimation} is dependent on. The occurring constants are dependent on the norms $ \Vert\mathbf{L}_\S^{-1}\mathbf{J}_\S\Vert_F$ and $ \Vert\mathbf{J}_\Y^T\mathbf{L}_\Y^{-T}\Vert_F$. Those terms are not dependent on the reduced dimensions $\hat s$ and $\hat q$ and can be computed as $ \Vert\mathbf{L}_\S^{-1}\mathbf{J}_\S\Vert_F\approx 2122.38$ and $ \Vert\mathbf{J}_\Y^T\mathbf{L}_\Y^{-T}\Vert_F \approx 2080.37$.
\end{example}

\section{Conclusion and outlook}

In this work, we have successfully derived an a-priori error estimate for the error between the full order solution obtained by space-time finite element discretization of a linear parabolic PDE and the reduced solution in the space-time reduced space. By doing so, we have gained insights about the convergence of the solution to the reduced system to the solution of the full order model. Furthermore, we have seen how the incorporation of an initial condition and the connected alteration of the optimal reduced space-time basis makes itself noticeable in the error estimate. We have investigated the theoretical error bound numerically and compared the practical appearance of the errors to the theoretical findings.

With the derivation of the error bound for the considered PDE, we aim to lay the foundation for similar error investigations for more advanced dynamics. The numerical consideration of the dynamics in Example~\ref{ex:circle} implies a rule for the choice of the reduced dimensions $\hat s$ and $\hat q $. The question arises, whether the gained insight can be extended as a general rule. Furthermore, we aim to relax the coercivity assumption in \eqref{eq:a_coercive_continuous}, using some technical standard arguments.

 For potential future research on the general topic of space-time POD or extending the space-time POD methodology, we refer to the open questions in \cite{baumann18}, which are still unanswered, e.g.\ the utilization of the underlying tensor structures for the treatment of parameter dependent systems or the extension of space-time POD to nonlinear systems, for which nonlinearities cannot be preassembled as in assumed in \cite{baumann18}.

\section*{Code availability}
The numerical examples in Section~\ref{sec:numerical examples} have been implemented in Python. The source code of the implementations used to create the results from Section~\ref{sec:numerical examples}, can be obtained from \url{https://doi.org/10.5281/zenodo.19710435}.

\section*{Author contributions} % see https://credit.niso.org 
\underline{Jan Heiland}: Resources; Software; Validation; Writing – review \& editing. \underline{Carmen Gräßle}: Conceptualization; Project administration; Supervision; Writing – review \& editing. \underline{Jannis Marquardt}: Conceptualization; Formal analysis; Investigation; Software; Visualization; Writing – original draft; Writing – review \& editing.

% Not used so far:

% Data curation – Management activities to annotate (produce metadata), scrub data and maintain research data (including software code, where it is necessary for interpreting the data itself) for initial use and later re-use.

% Funding acquisition - Acquisition of the financial support for the project leading to this publication.

% Methodology – Development or design of methodology; creation of models.

\bibliographystyle{scientific-computing}
\bibliography{main}

\end{document}